\pdfoutput=1
\documentclass[11pt,oneside]{amsart}
\usepackage{decoupling}
\usepackage[style=alphabetic,maxalphanames=5]{biblatex}
\hypersetup{
pdftitle={Decoupling inequalities for quadratic forms},
pdfauthor={Guo, Oh, Zhang, Zorin-Kranich},
}

\title{Decoupling inequalities for quadratic forms}
\author{Shaoming Guo}
\address[SG]{Department of Mathematics\\ University of Wisconsin Madison\\ USA}
\email{shaomingguo@math.wisc.edu}
\author{Changkeun Oh}
\address[CO]{Department of Mathematics\\ University of Wisconsin Madison\\ USA}
\email{coh28@wisc.edu}
\author{Ruixiang Zhang}
\address[RZ]{Department of Mathematics\\ University of California, Berkeley\\ USA}
\email{ruixiang@berkeley.edu}
\author{Pavel Zorin-Kranich}
\address[PZ]{Mathematical Institute\\ University of Bonn\\ Germany}
\email{pzorin@math.uni-bonn.de}
\subjclass[2020]{42B25 (Primary) 11L15, 26D05 (Secondary)}

\begin{document}

\begin{abstract}
We prove sharp $\ell^q L^p$ decoupling inequalities for $p,q \in [2,\infty)$ and arbitrary tuples of quadratic forms.
Connections to prior results on decoupling inequalities for quadratic forms are also explained.
We also include some applications of our results to exponential sum estimates and to Fourier restriction estimates.
The proof of our main result is based on scale-dependent Brascamp--Lieb inequalities.
\end{abstract}
\maketitle

\section{Introduction}
\label{sec:introduction}

Let $n, d \geq 1$.
We denote by $\bq(\xi)=(Q_1(\xi),\ldots,Q_n(\xi))$ an $n$-tuple of real quadratic forms in $d$ variables.
The graph of such a tuple, $S_{\bq}=\Set{(\xi,\bq(\xi)) \in [0, 1]^d \times \R^{n} }$, is a $d$-dimensional submanifold of $\R^{d+n}$.
We often write a spatial vector in $\R^{d+n}$ as $(x, y)$ with $x=(x_1, \dots, x_d)\in \R^d$ and $y=(y_1, \dots, y_n)\in \R^n$.
Similarly we often write a frequency vector in $ \R^{d+n}$ as $(\xi, \eta)$ with $\xi=(\xi_1, \dots, \xi_d)\in \R^d$ and $\eta=(\eta_{1}, \dots, \eta_{n})\in \R^n$.

Let $\Box\subset [0, 1]^d$.
Define the Fourier extension operator
\begin{equation}
E^{\bq}_{\Box} g(x, y):=\int_{\Box} g(\xi) e^{i(x\cdot \xi+y\cdot \bq(\xi))}d\xi,
\end{equation}
with $x\in \R^d, y\in \R^n$.
For $q, p\ge 2$ and dyadic $\delta\in (0, 1)$, let $D_{q, p}(\bq, \delta)$ be the smallest constant $D$ such that
\begin{equation}\label{201028e1.2}
\norm[\big]{E^{\bq}_{[0, 1]^d} g}_{L^p(w_{B})}\le D\Big( \sum_{\substack{\Box\subset [0, 1]^d\\ l(\Box)=\delta}} \norm[\big]{ E^{\bq}_{\Box} g}_{L^p(w_B)}^q\Big)^{1/q}
\end{equation}
holds for every measurable function $g$ and every ball $B\subset \R^{d+n}$ of radius $\delta^{-2}$, where $w_B$ is a smooth version of the indicator function of $B$ (see \eqref{201105e1_27} in the subsection of notation) and the sum on the right hand side runs through all dyadic cubes of side length $\delta$.
In the current paper, we determine an optimal asymptotic behavior of $D_{q, p}(\bq, \delta)$ as $\delta$ tends to zero, for every choice of $q, p\ge 2$ and $\bq$.

Before stating our main theorem, let us first review related results in the literature.
Decoupling theory originated from the work of Wolff \cite{MR1800068} and was further developed by \L aba and Wolff \cite{MR1956533}, \L aba and Pramanik \cite{MR2264215}, Garrig\'os and Seeger \cite{MR2546636} \cite{MR2664568} and Bourgain \cite{MR3038558}.
A breakthrough came with the resolution of the $\ell^2$-decoupling conjecture for paraboloids by Bourgain and Demeter \cite{MR3374964}.
Subsequently, Bourgain, Demeter and Guth \cite{MR3548534} resolved the Main Conjecture in Vinogradov's mean value theorem using decoupling theory.
We also refer to Bourgain and Demeter \cite{MR3614930}, Bourgain, Demeter and Guo \cite{MR3709122}, Guo and Zhang \cite{MR3994585} and Guo and Zorin-Kranich \cite{MR4031117} for extensions of \cite{MR3548534} to higher dimensions.

In the current paper, we study sharp decoupling inequalities for quadratic $d$-surfaces in $\R^{d+n}$ with $d, n\ge 1$.
The cases $n=1, d\ge 1$, that is, quadratic hypersurfaces, were the objects studied in Bourgain and Demeter \cite{MR3374964}, \cite{MR3736493}.
Since these works, there have been a number of other works studying sharp decoupling inequalities for quadratic $d$-surfaces in $\R^{d+n}$ with $n\ge 2$, that is, manifolds of co-dimension greater than one.
Bourgain's improvement on the Lindel\"of hypothesis \cite{MR3556291} relies on a decoupling inequality in the case $d=n=2$, which was later generalized and extended to a more general family of manifolds with dimension and codimension $2$ in \cite{MR3447712}.
Further sharp decoupling inequalities for (non-degenerate) quadratic $d$-surfaces of co-dimension $2$ were proven, for $2\le d\le 4$, in \cite{MR3945730} and \cite{MR4143735}.
More recently, in \cite{arxiv:1912.03995}, the classification of sharp decoupling inequalities for quadratic $3$-surfaces in $\R^5$ was completed, and sharp decoupling inequalities were proved in the ``degenerate'' cases, which were not covered by previously mentioned works.
The approach to the ``degenerate'' cases in \cite{arxiv:1912.03995} stands out from the previously mentioned works, in that it relies on small cap decouplings for the parabola and the $2$-surface $(\xi_1, \xi_2, \xi_1^2, \xi_1\xi_2)$ (we refer also to \cite{MR4153908} for further discussion of small cap decouplings).
For manifolds of co-dimension $n>2$, the only result for quadratic forms that are not monomials prior to the current article was due to Oh \cite{MR3848437}, who proved sharp decoupling inequalities for non-degenerate $3$-surfaces in $\R^6$.
In the current paper, we provide a unified approach that takes care of all the above examples, and indeed all quadratic $d$-surfaces in $\R^{d+n}$ for arbitrary combinations of $d$ and $n$.
In Section~\ref{sec:examples} we will explain in more detail how the above mentioned results fit into our main theorem.

Before we finish reviewing the related decoupling literature, we would also like to mention that Li \cite{MR4134240,MR4276295}, building partially on \cite{MR3939285,2020arXiv200907953G}, quantified the $\epsilon$ loss implicit in \eqref{bestconstant} in the case of the parabola.

Beyond decoupling theory, problems associated with quadratic $d$-surfaces ($d\ge 2$) of co-dimension bigger than one have also attracted much attention, in particular in Fourier restriction theory and related areas.
We refer to Christ \cite{MR766216}, \cite{christthesis}, Mockenhaupt \cite{Mockenhaupt}, Bak and Lee \cite{MR2064058}, Bak, Lee and Lee \cite{MR3694011}, Lee and Lee \cite{MR4319402}, Guo and Oh \cite{2020arXiv200907244G} for the restriction problems associated with manifolds of co-dimension two and higher, Bourgain \cite{MR1097257}, Rogers \cite{MR2231112} and Oberlin \cite{MR2289622} for the planar variant of the Kakeya problem, and Oberlin \cite{MR2078263} for sharp $L^p\to L^q$ improving estimates for a quadratic $3$-surface in $\R^5$.
Recently, Gressman \cite{MR3922044}, \cite{MR3992033} and \cite{MR4271482} has made significant progress in proving sharp $L^p$-improving estimates for Radon transforms of intermediate dimensions.
Perhaps more interestingly, he connected this problem with Brascamp-Lieb inequalities and geometric invariant theory.

One major difficulty in the development of the above mentioned problems in the setting of co-dimension bigger than one is the lack of a good notion of ``curvature''.
This is in strong contrast with the case of co-dimension one, where we have Gaussian curvatures and the notion of rotational curvatures, introduced by Phong and Stein in \cite{MR857680} and \cite{MR853446}.

Next let us turn to the main theorem of the current paper.
We will formulate a slightly more general (and essentially equivalent) version of \eqref{201028e1.2}.
This version uses functions with Fourier supports in small neighborhoods of $S_{\bq}$, instead of Fourier extension operators, and lends itself more readily to induction on dimension $d$.

It is convenient to define the Fourier supports in terms of symmetries of $S_{\bq}$.
The group $\bfA$ generated by translations and scalings of $\R^{d}$ consists of affine maps of the form $A(\xi) = \delta\xi + a$ with $a\in \R^{d}$ and $\delta\in (0,\infty)$ (in particular, $\bfA \cong \R^{d} \rtimes (0,\infty)$).
This group acts on $\R^{d+n}$ by affine transformations
\begin{equation}
\label{eq:affine-action-Rd+n}
\calA(A)(\xi,\eta)
= (\delta\xi + a, \delta^{2}\eta + \delta \nabla\bq(a) \cdot \xi + \bq(a)).
\end{equation}
This $\bfA$-action leaves $S_{\bq}$ invariant.
For a cube $\Box\subset\R^{d}$, let $A_{\Box}\in\bfA$ be the map such that $A_{\Box}([0,1]^{d})=\Box$, and denote the corresponding affine transformation on $\R^{d+n}$ by $\calA_{\Box}:=\calA(A_{\Box})$.
We define the associated uncertainty region by
\begin{equation}
\label{eq:Uncertainty-box}
\calU_{\Box} = \calU_{\Box}(\bq) :=
\calA_{\Box}\Bigl( [-2,2]^{d} \times \prod_{j=1}^{n} 4d (\norm{\Hess Q_{j}}+1) [-1,1] \Bigr).
\end{equation}
The main feature of the definition \eqref{eq:Uncertainty-box} is that the uncertainty region $\calU_{\Box}$ contains the convex hull of the graph of $\bq$ on $\Box$ and is not much larger than this convex hull.
Another convenient property is that
\begin{equation}
\label{eq:Uncertainty-box-nested}
2\Box \subseteq 2\Box'
\implies
\calU_{\Box} \subseteq \calU_{\Box'}.
\end{equation}
We will denote by $f_{\Box}$ an arbitrary function with $\supp \widehat{f_{\Box}} \subseteq \calU_{\Box}$.\\

Let $q,p \ge 2$.
Let $\delta<1$ be a dyadic number.
Let $\Part{\delta}$ be the partition of $[0, 1]^d$ into dyadic cubes of side length $\delta$.
Let $\Dec_{q, p}(\bq, \delta)$ be the smallest constant $D$ such that
\begin{equation}\label{equ:decoupling_definition}
\norm{\sum_{\Box\in \Part{\delta}} f_{\Box}}_{L^{p}(\R^{d+n})}
\le
D \Big(\sum_{\Box\in \Part{\delta}} \norm{f_{\Box}}_{L^{p}(\R^{d+n})}^q\Big)^{1/q}
\end{equation}
holds for every $f_{\Box}$ with $\supp \widehat{f}_{\Box} \subseteq \calU_{\Box}$.
If $p=q$, we often write
\begin{equation}
\Dec_p(\bq, \delta):=\Dec_{q, p}(\bq, \delta).
\end{equation}
Let $\Gamma_{q, p}(\bq)$ be the smallest constant $\Gamma$ such that, for every $\epsilon>0$, we have
\begin{equation}\label{bestconstant}
\Dec_{q, p}(\bq, \delta)\le C_{p, q, \bq, \epsilon} \delta^{-\Gamma-\epsilon}, \text{ for every dyadic } \delta<1,
\end{equation}
where $C_{p, q, \bq, \epsilon}$ is a constant that is allowed to depend on $p, q, \bq$ and $\epsilon$.
If $p=q$, we often write
\begin{equation}
\Gamma_p(\bq):=\Gamma_{p, p}(\bq).
\end{equation}
For a tuple $\widetilde{\bq}=(\widetilde{Q}_1(\xi), \dots, \widetilde{Q}_{\tilde{n}}(\xi))$ of quadratic forms with $\xi\in \R^d$, denote
\begin{equation}
\NV(\widetilde{\bq}):=|\Set{1\le d'\le d: \partial_{\xi_{d'}} \widetilde{Q}_{\tilde{n}'}\not\equiv 0 \text{ for some } 1\le \tilde{n}'\le \tilde{n}}|.
\end{equation}
Here for a function $F$, we use $F\not\equiv 0$ to mean that it does not vanish constantly, and $\NV(\widetilde{\bq})$ refers to ``the number of variables that $\widetilde{\bq}$ depends on''.
For instance, for $\widetilde{\bq}=((\xi_1+\xi_3)^2, (\xi_1+\xi_3+\xi_4)^2)$, we have that $\NV(\widetilde{\bq})=3$.

For $0\le n'\le n$ and $0 \leq d' \leq d$, define
\begin{equation}
\label{eq:numvar}
\numvar_{d',n'}(\bq):=
\inf_{\substack{M\in \R^{d\times d}\\ \rank(M)=d'}}
\inf_{\substack{M'\in \R^{n\times n}\\ \rank(M')=n'}}
\NV(M' \cdot (\bq \circ M)),
\end{equation}
where $\bq\circ M$ is the composition of $\bq$ with $M$.
We abbreviate $\numvar_{n'}(\bq) := \numvar_{d,n'}(\bq)$.

As these quantities will be crucial throughout the entire paper, it may make sense to look at them from a few angles.
Let $H\subset \R^d$ be a linear subspace of codimension $m$.
Let $\rot_H$ be a rotation\footnote{There are infinitely many such rotations: We pick an arbitrary one.} on $\R^d$ that maps $\Set{\xi\in \R^d: \xi_{d-m+1}=\dots=\xi_d=0}$ to $H$.
We define
\begin{equation}
\label{eq:Q-on-subspace}
Q|_H(\xi'):=Q((\xi', {\bf 0})\cdot (\rot_H)^T) , \text{ with } \xi'\in \R^{d-m} \text{ and } {\bf 0}=(0, \dots, 0)\in \R^m.
\end{equation}
Here $(\rot_H)^T$ refers to the transpose of $\rot_H$.
Similarly, for $\bq=(Q_1, \dots, Q_n)$, we denote
\begin{equation}
\bq|_H:=(Q_1|_H, \dots, Q_n|_H).
\end{equation}
From the Bourgain-Guth argument used in Proposition~\ref{prop:Bourgain-Guth} below and from Lemma~\ref{lem:subspace-lower-bd} below one can see clearly why restricting to subspaces is natural.
With the above notation, we can also write
\begin{equation}\label{201105e1_15}
\numvar_{d', n'}(\bq)=\inf_{H}\inf_{M\in \text{GL}_{d'}} \inf_{\substack{M'\in \R^{n'\times n}\\ \rank(M')=n'}} \NV(M'\cdot (\bq|_H\circ M)),
\end{equation}
where $H$ runs through all linear sub-spaces in $\R^d$ of dimension $d'$.
In other words, this is the minimal number of variables that $n'$ many of the forms in $\bq$ depend on, after restricting them to sub-spaces of dimension $d'$, up to linear changes of variables in their definition domain $\R^{d'}$ and their value domain $\R^{n}$. From \eqref{201105e1_15}, one can also see that
\begin{equation}\label{201107e1_16}
\numvar_{d', n'}(\bq)=\inf_{H \text{ of dim } d'} \numvar_{d', n'}(\bq|_H).
\end{equation}
For example, if we take $d=3, n=3$ and $\bq=((\xi_1+\xi_2+\xi_3)^2, (\xi_1+\xi_2)^2, (\xi_1+\xi_2)^2)$, then $\numvar_0(\bq)=\numvar_1(\bq)=0$, $\numvar_2(\bq)=1$ and $\numvar_3(\bq)=2$.
\begin{theorem}
\label{thm:main}
Let $d\ge 1$ and $n\ge 1$.
Let $\bq=(Q_1, \dotsc, Q_n)$ be a collection of quadratic forms in $d$ variables.
Let $2 \leq q \leq p < \infty$.
Then the $\ell^{q}L^{p}$ decoupling exponent for the $d$-surface $S_{\bq}=\Set{(\xi, \bq(\xi)): \xi\in [0, 1]^d}$ equals
\begin{equation}
\label{eq:dec-exponent-unrolled}
\Gamma_{q,p}(\bq) =
\max_{0\leq d'\leq d} \max_{0\leq n'\leq n} \Big( d' \bigl(1-\frac{1}{p}-\frac{1}{q} \bigr) - \numvar_{d',n'}(\bq)\bigl(\frac{1}{2}-\frac{1}{p}\bigr) - \frac{2(n-n')}{p}\Big).
\end{equation}
Moreover, for $2 \leq p < q \leq \infty$, we have
\begin{equation}
\label{eq:dec-exponent-q>p}
\Gamma_{q,p}(\bq) = \Gamma_{p,p}(\bq) + d(1/p-1/q).
\end{equation}
\end{theorem}

The expression \eqref{eq:dec-exponent-unrolled} simplifies in the case $q=p$ in the following way.

\begin{corollary}\label{1015.1.3}
In the situation of Theorem~\ref{thm:main}, we have
\begin{equation}
\label{eq:dec-exp-ell-p-L-p}
\Gamma_{p}(\bq) =
\max_{d/2 < d'\leq d} \max_{0\leq n'\leq n} \Big( (2d' - \numvar_{d',n'}(\bq)) \bigl(\frac{1}{2}-\frac{1}{p}\bigr) - \frac{2(n-n')}{p}\Big).
\end{equation}
for every $p\ge 2$.
\end{corollary}
\begin{proof}
For every $d' \leq d/2$ and $0 \leq n' \leq n$, we have $2d'-\numvar_{d',n'}(\bq)\leq d \leq 2d-\numvar_{d,n'}(\bq)$.
Hence, the $(d',n')$ term in \eqref{eq:dec-exponent-unrolled} is not larger than the $(d,n')$ term.
\end{proof}

Taking $d'=d$ and $n' \in \Set{0,n}$ in \eqref{eq:dec-exp-ell-p-L-p}, we see that, for every tuple $\bq=(Q_1, \dotsc, Q_n)$ of quadratic forms depending on $d$ variables, it always holds that
\begin{equation}
\label{eq:Gamma-p-p-universal-lower-bd}
\Gamma_{p}(\bq)\ge
\max{\Big(d(\frac12-\frac1p),2d(\frac12-\frac1p)-\frac{2n}{p} \Big)}
\text{ for every } p\ge 2.
\end{equation}
Similarly,
\begin{equation}
\label{eq:Gamma-2-p-universal-lower-bd}
\Gamma_{2,p}(\bq)\ge
\max{\Big(0,d(\frac12-\frac1p)-\frac{2n}{p} \Big)}
\text{ for every } p\ge 2.
\end{equation}
We say that $\bq=(Q_1,\ldots,Q_n)$ is \emph{strongly non-degenerate} if
\begin{equation}\label{equ:1strongly_non_dege}
\numvar_{d-m,n'}(\bq) \geq n' d/n-m,
\end{equation}
for every $n'$ and every $m$ with $0 \leq m \leq d$.

\begin{corollary}[Best possible $\ell^{2}L^{p}$ decoupling]\label{optimal-l2-decoupling}
We have
\begin{equation}\label{20201022.1.4}
\Gamma_{2,p}(\bq) =
\max{\Big(0,d(\frac12-\frac1p)-\frac{2n}{p} \Big)}
\text{ for every } 2 \leq p < \infty
\end{equation}
if and only if $\bq$ is strongly non-degenerate.
\end{corollary}

We say that $\bq=(Q_1,\ldots,Q_n)$ is \emph{non-degenerate} if
\begin{equation}\label{eq:non_degenerate}
\numvar_{d-m,n'}(\bq) \geq n' d/n-2m,
\end{equation}
for every $n'$ and every $m$ with $0 \leq m < d/2$.

\begin{corollary}[Best possible $\ell^{p}L^{p}$ decoupling]
\label{201017coro1.3}
We have
\begin{equation}\label{20201021.1.4}
\Gamma_{p,p}(\bq)=\max{\Big(d(\frac12-\frac1p),2d(\frac12-\frac1p)-\frac{2n}{p} \Big)}
\text{ for every } 2 \leq p < \infty
\end{equation}
if and only if $\bq$ is non-degenerate.
\end{corollary}
In view of \eqref{eq:Gamma-2-p-universal-lower-bd} and \eqref{eq:Gamma-p-p-universal-lower-bd}, Corollary~\ref{optimal-l2-decoupling} and Corollary~\ref{201017coro1.3} characterize tuples of quadratic forms that possess ``best possible'' $\ell^2 L^p$ decoupling constants and $\ell^p L^p$ decoupling constants, respectively.

We say that $\bq=(Q_1,\ldots,Q_n)$ is \emph{weakly non-degenerate} if
\begin{equation}\label{equ:weakly_non_dege}
\numvar_{d-m,n}(\bq) \geq d-2m,
\end{equation}
for every $0 \leq m < d/2$.

\begin{corollary}\label{20201021.1.5.}
A tuple $\bq=(Q_1,\ldots,Q_n)$ of quadratic forms is weakly non-degenerate if and only if there exists some $p_c>2$ such that
\begin{equation}\label{20201021.1.21.}
\Gamma_p(\bq)=d(\frac12-\frac1p), \;\;\;\; 2 \leq p \leq p_c.
\end{equation}
If $\bq$ is weakly non-degenerate, then the largest possible $p_c$ is given by
\begin{equation}\label{giantnumber}
2+\min
\biggl( \frac{4(n-n')}{d-\bigl(\numvar_{d-m,n'}(\bq)+2m\bigr) }\biggr),
\end{equation}
where the minimum on the right hand side is taken over all $n'$ and $m$ satisfying $n'\le n-1, m< d/2$ and $d>\numvar_{d-m,n'}(\bq)+2m$.
\end{corollary}
One reason that we are interested in the exponent $p_c$ in Corollary~\ref{20201021.1.5.} is that, when applying our main results to exponential sum estimates (see Corollary~\ref{211127coro2.1.} below), the exponent $p_c$ is the largest for which we can still expect square root cancellation; see right below Corollary~\ref{211127coro2.1.} for what we mean by square root cancellation.

The connections of these three notions of non-degeneracies will be discussed in forthcoming examples, see for instance Corollary~\ref{201024coro3.5}.
We leave the proof of Corollary~\ref{optimal-l2-decoupling}--\ref{20201021.1.5.} to Section~\ref{201031section7}.
\medskip

At the end of the introduction, we would like to make a few remarks on Theorem~\ref{thm:main} and how the double $\max$ on the right hand side of \eqref{eq:dec-exponent-unrolled} is connected to different regimes in scale-dependent Brascamp--Lieb inequalities (see Theorem~\ref{thm:BL} below).
To enable comparison with previous approaches, we specialize to $p=q$, so that \eqref{eq:dec-exponent-unrolled} simplifies to \eqref{eq:dec-exp-ell-p-L-p}, and consider the tuples of quadratic forms
\begin{equation}\label{201112e1_28}
\bq_1(\xi)=(\xi_1\xi_2, \xi_1\xi_3, \xi_2\xi_3, \xi_3^2),
\quad
\bq_2(\xi)=(\xi_1^2, \xi_2^2+\xi_1\xi_3),
\end{equation}
for which sharp $\ell^{p}L^{p}$ decoupling inequalities were previously proved in \cite{MR4031117} and \cite{arxiv:1912.03995}, respectively.
The reasons of picking these two examples will soon become clear.
The numbers of variables \eqref{eq:numvar} that appear in \eqref{eq:dec-exp-ell-p-L-p} in these cases are summarized in Table~\ref{tab:numvar}.
\begin{table}
\begin{tabular}{l@{\hspace{1cm}}m{0.6cm}m{0.6cm}m{0.6cm}m{0.6cm}m{0.6cm}@{\hspace{1cm}}m{0.6cm}m{0.6cm}m{0.6cm}}
\toprule
& \multicolumn{5}{c@{\hspace{1cm}}}{$\numvar_{d',n'}(\bq_{1})$} & \multicolumn{3}{c@{\hspace{1cm}}}{$\numvar_{d',n'}(\bq_{2})$} \\
$n'=$ & 0 & 1 & 2 & 3 & 4 & 0 & 1 & 2\\ \midrule
$d'=3$ & 0 & 1 & 2 & 3 & 3 & 0 & 1 & 3\\
$d'=2$ & 0 & 0 & 0 & 0 & 2 & 0 & 0 & 1\\ \bottomrule
\end{tabular}
\vspace{1ex}
\caption{\label{tab:numvar}Minimal number of variables on which lower dimensional restrictions of $\bq_{1}$ and $\bq_{2}$ depend.}
\end{table}

To illustrate the kind of arguments used to obtain the entries in Table~\ref{tab:numvar}, let us consider $\numvar_{3,3}(\bq_{1})$.
If $\numvar_{3,3}(\bq_{1}) \leq 2$, then there is a $3$-dimensional subspace $\calQ$ of the linear space of quadratic forms spanned by $\bq_{1}$, such that all forms in the subspace $\calQ$ depend on at most 2 variables.
The space of quadratic forms depending on any 2 variables has dimension 3, so $\calQ$ consists of all quadratic forms depending on these 2 variables.
In particular, $\calQ$ contains 2 linearly independent quadratic forms that are complete squares.
Therefore, $\lin \bq_{1}$ constants a quadratic form of the form $(a\xi_{1}+b\xi_{2}+c\xi_{3})^{2}$ with $(a,b)\neq (0,0)$.
But such a quadratic form includes one of the monomials $\xi_{1}^{2},\xi_{2}^{2}$ with a non-zero coefficient, contradicting the fact that it lies in $\lin \bq_{1}$.
Hence, $\numvar_{3,3}(\bq_{1}) > 2$, and since $\numvar_{3,3}(\bq_{1}) \leq 3$, we obtain $\numvar_{3,3}(\bq_{1}) = 3$.
Upper bounds for $\numvar_{d',n'}$ are usually easier to obtain.
For example, with the notation from \eqref{201105e1_15}, the upper bounds for $\numvar_{2,n'}(\bq_{1})$ are obtained with $H=\Set{\xi_{3}=0}$ and the upper bounds for $\numvar_{2,n'}(\bq_{2})$ with $H=\Set{\xi_{1}=0}$.

Substituting the numbers in Table~\ref{tab:numvar} into \eqref{eq:dec-exp-ell-p-L-p}, we see that the decoupling exponents are given by
\begin{align}
\label{eq:Gamma-example-1}
\Gamma_p(\bq_1) &= \max\Bigl(\frac 3 2-\frac 3 p, 2-\frac 6 p, 3-\frac{14}{p} \Bigr) \quad\text{and}\\
\label{eq:Gamma-example-2}
\Gamma_p(\bq_2) &= \max\Bigl(\frac 3 2-\frac 3 p, \frac 5 2-\frac 7 p, 3-\frac{10}{p} \Bigr)
\end{align}
for every $p\geq 2$.
These decoupling exponents are sketched as functions of $1/p$ in Figure~\ref{fig:q1+q2} (not to scale).
\begin{figure*}[t!]
\captionsetup[subfigure]{justification=centering}
\centering
\begin{subfigure}[t]{0.5\textwidth}
\centering
\includegraphics[width=\textwidth]{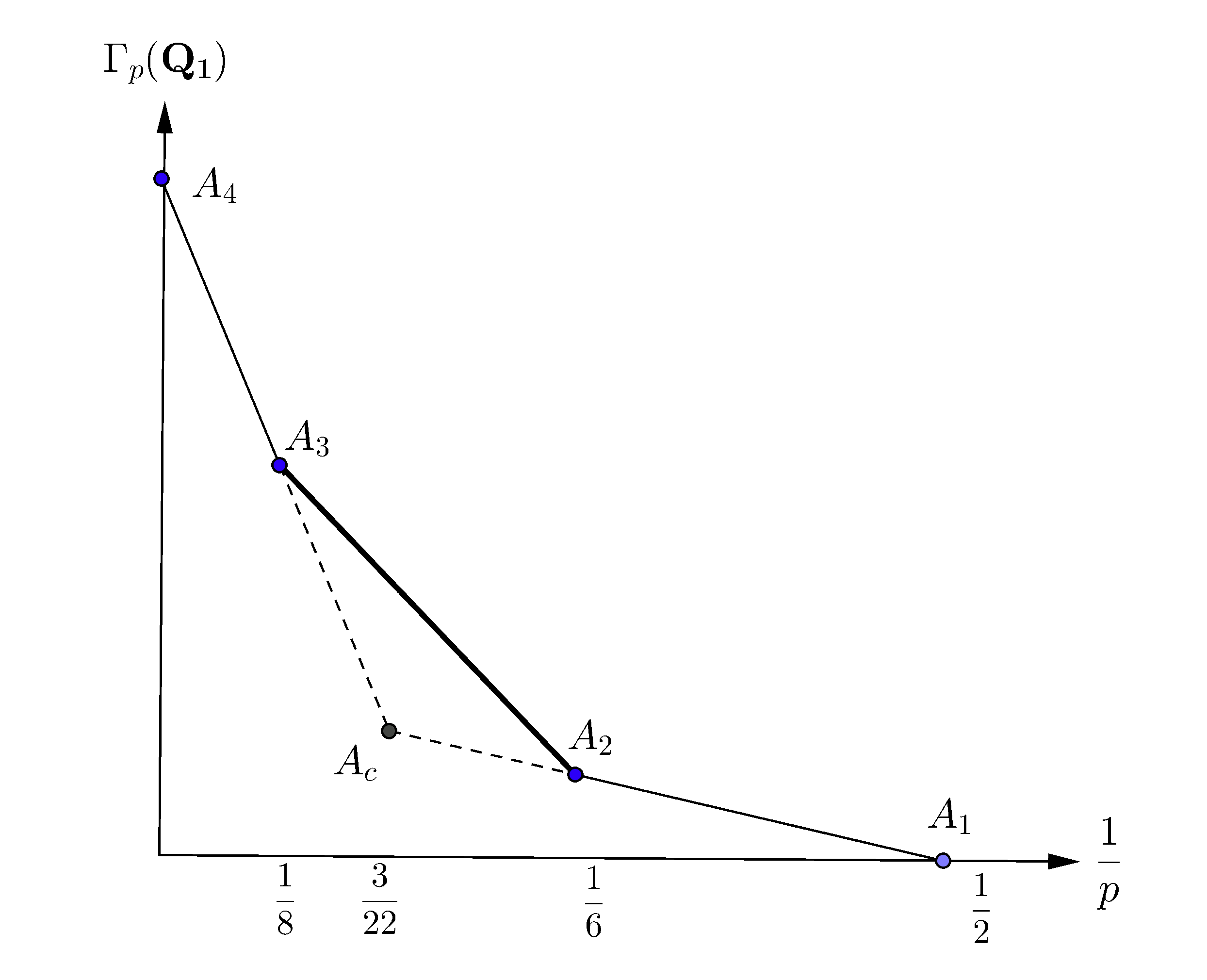}
\caption{$\bq_{1}(\xi)=(\xi_{1}\xi_{2},\xi_{1}\xi_{3},\xi_{2}\xi_{3},\xi_{3}^{2})$}
\label{fig:q1}
\end{subfigure}%
~ 
\begin{subfigure}[t]{0.5\textwidth}
\centering
\includegraphics[width=\textwidth]{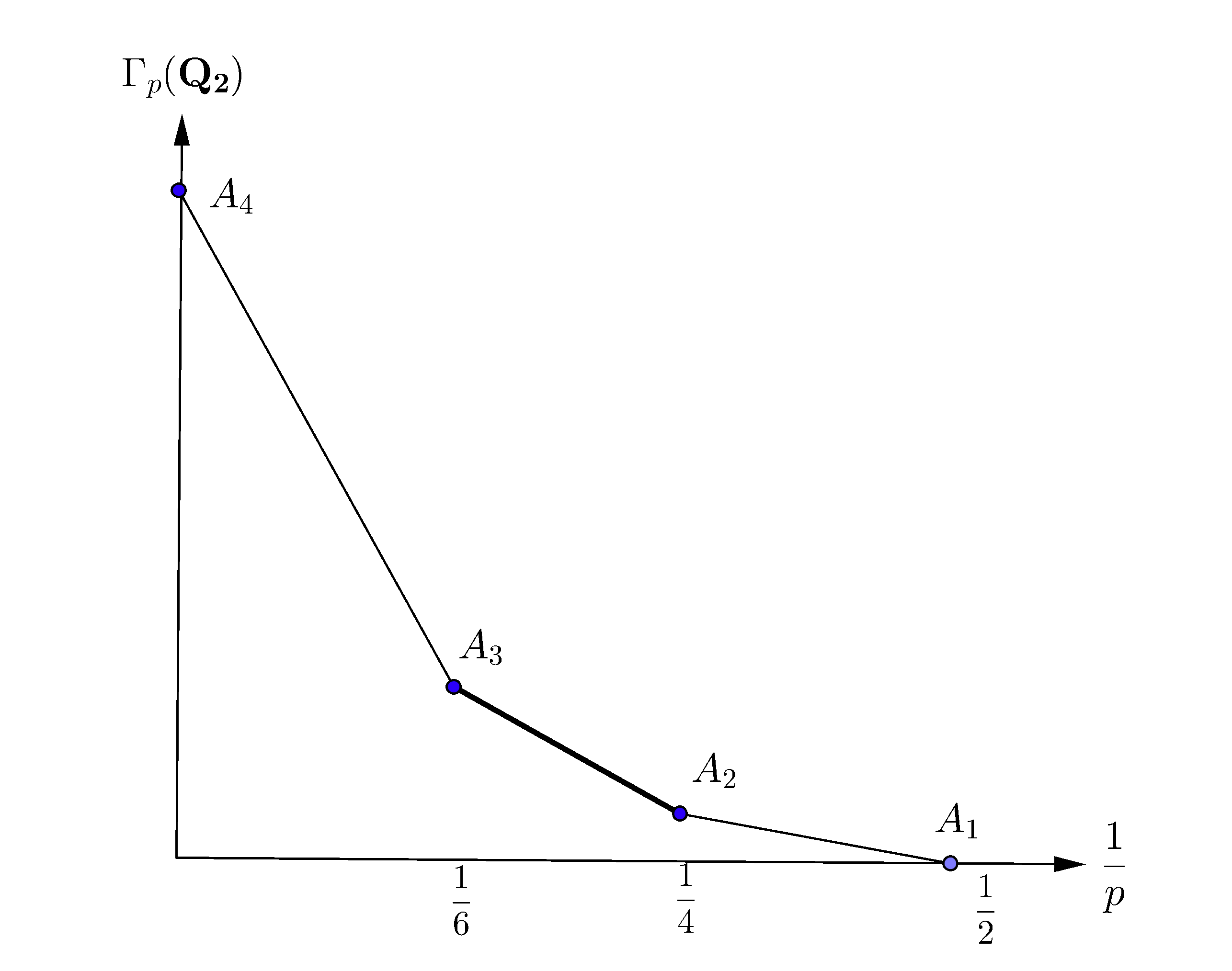}
\caption{$\bq_{2}(\xi)=(\xi_{1}^{2},\xi_{2}^{2}+\xi_{1}\xi_{3})$}
\label{fig:q2}
\end{subfigure}
\caption{Decoupling exponents for $\bq_{1},\bq_{2}$.}
\label{fig:q1+q2}
\end{figure*}
The kinks $A_2,A_3$ in the graph \ref{fig:q1} and $B_{2},B_{3}$ in the graph \ref{fig:q2} are called critical points: After proving sharp decoupling inequalities at these two points, we can simply interpolate them with trivial estimates at $L^2$ and $L^{\infty}$ and obtain sharp $\ell^{p}L^{p}$ decoupling inequalities for every $p\ge 2$.
One reason for picking these two examples is: On one hand, Graph \ref{fig:q1} and Graph \ref{fig:q2} look similar, in that they both have two critical points; on the other hand, the nature of the critical points in the respective graphs is entirely different, and this can sometimes be misleading when it comes to trying to come up with a unified proof strategy for both examples.
In the following several paragraphs, let us try to explain the different nature of the critical points, together with one main ingredient of the paper---scale-dependent Brascamp-Lieb inequalities.

The tuple $\bq_{1}$, together with the linear forms $\xi_1, \xi_2, \xi_3$, is the Arkhipov-Chubarikov-Karatsuba system (ACK system in short) generated by the monomial $\xi_1 \xi_2\xi_3^2$, restricted to the degree $2$.
To be more precise, we take all possible partial derivatives of $\xi_1 \xi_2\xi_3^2$, collect the resulting terms, and throw away the terms that are of degree higher than 2.
Decoupling exponents for all ACK systems were found in \cite{MR4031117}.
The main argument to handle $\bq_1$ there was purely guided by the ``fake'' kink $A_{c}$ and the critical points $A_2$ and $A_3$ were deliberately neglected; see \cite[Section 7]{MR3994585} for a detailed discussion on the critical points and the ``fake" kink in Parsell--Vinogradov systems.
Parsell-Vinogradov systems are special cases of ACK systems and that discussion applies equally well to ACK systems.
The example \eqref{201112e1_28} is perhaps the simplest system of quadratic monomials that admits more than one critical point, and is therefore discussed here in detail.%
The way $A_c$ appears in \cite{MR4031117} can be summarized as follows: For all $p\ge 2$, the article \cite{MR4031117} uses only the global/scale-invariant Brascamp-Lieb inequality due to Bennett, Carbery, Christ and Tao \cite{MR2377493}; the difference between $p>p_c:=22/3$ and $p<p_c$ is in the scheme of induction on scales there.

Let us describe the current argument of handling $\bq_1$.
In the current paper, we still use the argument of induction on scales (see Section~\ref{sec:induction-on-scales}), and a phase transition also happens at the ``fake" kink $A_c$.
However, for $p \leq p_{c}$, we will use local Brascamp-Lieb inequalities as in \cite[Theorem 2.2]{MR2661170}, and for $p \geq p_{c}$, we will use discrete Brascamp-Lieb inequalities as in \cite[Theorem 2.5]{MR2661170}.
For $p=p_{c}$, local and discrete Brascamp-Lieb inequalities coincide and become scale-invariant, as in \cite[Theorem 1.15]{MR2377493}.
The significance of local and discrete Brascamp-Lieb inequalities may not be fully reflected from this example, as they coincide at some point.
The next example explains what we could do in a case where they do not coincide at any point.

Let us turn to the tuple $\bq_{2}$.
As mentioned above, sharp decoupling inequalities for $\bq_2$ were proven in \cite{arxiv:1912.03995}.
The tuple $\bq_2$ is an interesting example as it is perhaps the simplest non-trivial example that is degenerate, in the sense that the scale-invariant Brascamp-Lieb inequality cannot be used for any $p \in [2,\infty)$.
To handle this degeneracy, several specialized tools were introduced in \cite{arxiv:1912.03995}, including a partial small ball decoupling inequality for the tuple $(\xi_1\xi_2, \xi_2^2)$.
In \cite{arxiv:1912.03995}, sharp decoupling inequalities were first proved at the critical points $B_2$ and $B_3$, and interpolation with trivial estimates at $B_1$ and $B_4$ was used to cover the whole range $p\ge 2$.

Let us describe our current approach.
First of all, the counterpart of $A_c$ in Graph \ref{fig:q2} is not significant anymore.
Secondly, it turns out that for $2 \leq p\leq 4$ we can still use the local Brascamp-Lieb inequality and for $p\geq 6$ the discrete Brascamp-Lieb inequality (it is a general principle that the local Brascamp-Lieb inequality works for small $p$ and the discrete one for large $p$, although the ranges of $p$'s in which they work may be empty).
There is, however, a new regime for $p \in (4,6)$.
Although this segment can be filled in by interpolation between $p=4$ and $p=6$, our proof works directly for all $p$, which becomes necessary when the number of kinks increases further.
To this end, we use the family of scale-dependent Brascamp-Lieb inequalities due to Maldague \cite{arxiv:1904.06450}, which unifies scale-invariant, discrete, and local Brascamp-Lieb inequalities due to Bennett, Carbery, Christ, and Tao \cite{MR2377493,MR2661170}.

\subsection*{Organization of the paper}
In Section~\ref{201031section2} we state a few applications of our main theorem.
In Section~\ref{sec:examples}, we compute the decoupling exponent provided by the main theorem more explicitly for several examples of tuples of quadratic forms $\bq$, including some of those tuples $\bq$ for which sharp decoupling inequalities have been previously established in the literature, and a few tuples $\bq$ (in particular, arbitrary pairs of forms and tuples of simultaneously diagonalizable forms) for which our results are new.

The upper bounds of $\Gamma_{q, p}(\bq)$ in Theorem~\ref{thm:main} with $q\le p$ are proven in Section~\ref{sec:transversality} and Section~\ref{sec:induction-on-scales}.
The lower bounds of $\Gamma_{q, p}(\bq)$ in Theorem~\ref{thm:main} with $q\le p$ are proven in Section~\ref{sec:lower-bd}.
In Section~\ref{201031section8}, we show that the optimal decoupling inequalities for $q>p$ follow from the case $q=p$ of Theorem~\ref{thm:main}.

In Section~\ref{201031section7}, we provide the proofs of Corollary~\ref{optimal-l2-decoupling}, Corollary~\ref{201017coro1.3} and Corollary~\ref{20201021.1.5.}.
In Section~\ref{201031section10}, we prove the Fourier restriction estimate in Corollary~\ref{coro:app_to_restriction}.

\subsection*{Notation}
For two positive constants $A_1, A_2$ and a set $\calI$ of parameters, we use $A_1\lesssim_{\calI} A_2$ to mean that there exists $C>0$ depending on the parameters in $\calI$ such that $A_1\le C A_2$.
Typically, $\calI$ will be taken to be $\Set{\bq, d, n, p, q, \epsilon}$ where $\epsilon>0$ is a small number.
Similarly, we define $A_1\gtrsim A_2$.
Moreover, $A_1\sim A_2$ means $A_1\lesssim A_2$ and $A_1\gtrsim A_2$.

Let $\delta\in (0, 1)$ be a dyadic number.
We denote by $\mc{P}(Q,\delta)$ the dyadic cubes of side length $\delta$ in $Q$ for every dyadic cube $Q \subset [0,1]^d$.
Let $\mc{P}(\delta)$ be the partition of $[0, 1]^d$ into dyadic cubes of side length $\delta$.
Let $\Box$ be a cube with side length $l(\Box)$, we use $c\cdot \Box$ to denote the cube of side length $c\cdot l(\Box)$ and of the same center as $\Box$.

For two linear spaces $V, V'$, we use $V'\leq V$ to mean that $V'$ is a linear subspace of $V$.
For a sequence of real numbers $\Set{A_j}_{j=1}^M$, we abbreviate $\avprod A_{j} := \bigl(\prod_{j=1}^{M} |A_{j}|\bigr)^{1/M}$.
For $E>0$ and a ball $B=B(c_B, r_B)\subset \R^{d+n}$ with center $c_B$ and radius $r_B$, define an associated weight
\begin{equation}\label{201105e1_27}
w_{B, E}(\cdot):=\Big(1+\frac{|\cdot-c_B|}{r_B}\Big)^{-E}.
\end{equation}
The power $E$ is large number depending on $d, n$, e.g., $E=10(d+n)$, and will be omitted from the notation $w_{B, E}$.
All implicit constants in the paper are allowed to depend on $E$.
Also, we define averaged integrals:
\[
\norm{f}_{\avL^{p}(B)}:=(\frac{1}{|B|}\int_B |f|^p )^{1/p} \text{ and } \norm{f}_{\avL^{p}(w_B)}:=(\frac{1}{|B|}\int |f|^p w_B)^{1/p}.
\]
For a dyadic box $\Box\subset [0, 1]^d$, a function $f_{\Box}$ is always implicitly assumed to satisfy $\supp \widehat{f}_{\Box}\subset \calU_{\Box}$, unless otherwise stated.

We would like to make the convention that all vectors are column vectors, unless they are variables of functions or otherwise stated.
Below are a few more conventions we make on notation: We will use dyadic cubes of side lengths $\delta,$ $\delta^b$ with $b<1$ and $1/K_j$ with $1\le j\le d$.
One can always keep in mind that $\log_{K_j}\log_{K_j} (1/\delta)\ge K_j$.
We will always use $\Box$ to denote a dyadic cube of the smallest scale $\delta$, $J$ to denote a dyadic cube of an intermediate scale $\delta^b$, and $W$ or $W_j$ to denote a dyadic cube of a large scale $1/K_j$.
We will introduce certain multi-linear estimates during the proof, and the degree of the multi-linearity will always be called $M$.

\subsection*{Acknowledgment}
The authors would like to thank the referees for reading this paper carefully and for their valuable suggestions. S.G.\ was supported in part by the NSF grants DMS-1800274 and DMS-2044828.
R.Z.\ was supported by the NSF grant DMS-1856541, DMS-1926686 and by the Ky Fan and Yu-Fen Fan Endowment Fund at the Institute for Advanced Study.
Part of the work was done during a visit of S.G.\ and R.Z.\ to Bonn in the summer of 2019, partially funded by the Deutsche Forschungsgemeinschaft (DFG, German Research Foundation) under Germany's Excellence Strategy -- EXC-2047/1 -- 390685813.
They would like to thank the Hausdorff Center for Mathematics for the hospitality during their stay.

\section{Some Applications}\label{201031section2}

\subsection{Exponential sum estimates.}

Let $\bq=(Q_1, \dots, Q_n)$ be a collection of quadratic forms of integral coefficients defined on $\R^d$.
Let $\bfww=(w_1, \dots, w_d)\in \N^d$.
\begin{corollary}\label{211127coro2.1.}
For every $d, n\ge 1$, every $p\ge 2, \epsilon>0$, there exists $C_{\bq, \epsilon, p}$ such that
\begin{equation}\label{201031e2_1}
\norm[\Big]{\sum_{1\le d'\le d}\sum_{0\le w_{d'}\le W} e^{2\pi i(\bfww\cdot x+ \bq(\bfww)\cdot y)}}_{L^p([0, 1]^d\times [0, 1]^n)}\le C_{\bq, \epsilon, p} W^{\Gamma_p(\bq)+\frac{d}{p}+\epsilon}
\end{equation}
for every integer $W$.
\end{corollary}

If $\Gamma_p(\bfQ)=d(1/2-1/p)$, then the above corollary says that
\begin{equation}\label{201031e2_1zzz}
\begin{split}
& \norm[\Big]{\sum_{1\le d'\le d}\sum_{0\le w_{d'}\le W} e^{2\pi i(\bfww\cdot x+ \bq(\bfww)\cdot y)}}_{L^p([0, 1]^d\times [0, 1]^n)}\\
& \le C_{\bq, \epsilon, p} W^{\epsilon} \norm[\Big]{\sum_{1\le d'\le d}\sum_{0\le w_{d'}\le W} e^{2\pi i(\bfww\cdot x+ \bq(\bfww)\cdot y)}}_{L^2([0, 1]^d\times [0, 1]^n)},
\end{split}
\end{equation}
by which we mean square root cancellation holds for the exponential sum at such $p$.

The derivation of exponential sum estimates of the form in the above corollary from decoupling inequalities has been standard, see for instance Section 2 \cite{MR3374964} and Section 4 \cite{MR3548534}.
We will not repeat the argument here.

Let $s\ge 1$ be an integer.
Consider the system of Diophantine equations
\begin{equation}\label{201023e2_1}
\begin{split}
& \bfww_1+\dots+\bfww_s=\bfww_{s+1}+\dots+\bfww_{2s},\\
& \bq(\bfww_1)+\dots+\bq(\bfww_s)=\bq(\bfww_{s+1})+\dots+\bq(\bfww_{2s}).
\end{split}
\end{equation}
For a large integer $W$, let $\mc{J}_{\bq}(W)$ be the number of solutions to \eqref{201023e2_1}
where $0\le w_{d'}\le W$ for every $d'$.
As a immediate corollary of \eqref{201031e2_1}, we obtain
\begin{corollary}
For every $d, n\ge 1$, integer $s\ge 1$, and every $\epsilon>0$, there exists $C_{\bq, \epsilon, s}$ such that
\begin{equation}
\mc{J}_{\bq}(W)\le C_{\bq, \epsilon, s} W^{2s \Gamma_{2s}(\bq)+d+\epsilon},
\end{equation}
for every $W$.
\end{corollary}

\subsection{Fourier restriction estimates.}
Let $\bq=(Q_1, \dots, Q_n)$ be a collection of quadratic forms defined on $\R^d$.
We say that $\bq$ is \textit{linearly independent} if $Q_1,\ldots,Q_n$ are linearly independent.
We are interested in the Fourier restriction problems: Find an optimal range of $p$ such that
\begin{equation}\label{1031restriction}
\norm{E_{[0, 1]^d}^{\bq}g}_{L^p(\R^{d+n})}\lesssim_{d, n, p, \bq} \norm{g}_{p}
\end{equation}
holds true for every function $g$.
By a simple change of variables, one can see that the restriction estimate \eqref{1031restriction} cannot hold true for any $p<\infty$ if $\bq$ is linearly dependent.
As an application of Corollary~\ref{1015.1.3}, we prove some restriction estimates for every linearly independent $\bq$ for some range of $p$.

\begin{corollary}\label{coro:app_to_restriction}
Let $\bq=(Q_1, \dots, Q_n)$ be a collection of linearly independent quadratic forms defined on $\R^d$.
Then
\begin{equation}\label{201031e2_5}
\norm{E_{[0, 1]^d}^{\bq}g}_{L^p(\R^{d+n})}\lesssim_{d, n, p, \bq} \norm{g}_{p}
\end{equation}
for every
\begin{equation}\label{201022e2_3}
p>p_{\bq}:=2+ \max_{m\ge 1}\max_{n'\le n}\frac{4n'}{2m+\numvar_{d-m, n'}(\bq)}.
\end{equation}
\end{corollary}
The proof of this corollary will be presented in Section~\ref{201031section10}.
One significance of this corollary is that the range \eqref{201022e2_3} is sharp for Parsell--Vinogradov systems.
Let us be more precise.
Let $d\ge 2$.
Denote $\xi^{\alpha}:=\xi_1^{\alpha_1}\dots \xi_d^{\alpha_d}$ for $\xi=(\xi_1, \dots, \xi_d)$ and a multi-index $\alpha=(\alpha_1, \dots, \alpha_d)$.
For $\bq:=(\xi^{\alpha})_{|\alpha|= 2}$, we have $n=d(d+1)/2$, and it has been shown by Christ \cite{christthesis} and Mockenhaupt \cite{Mockenhaupt} that \eqref{201031e2_5} holds if and only if
\begin{equation}\label{eq:9}
p>2+\frac{4n}{d+1}=2d+2.
\end{equation}
Let us also mention that, for this tuple $\bq$, the full range of $L^{q}\to L^{p}$ estimates generalizing \eqref{201031e2_5} has been obtained in \cite{MR2064058,MR2137104}.
The next claim shows that the range \eqref{201022e2_3} coincides with \eqref{eq:9}.
\begin{claim}\label{201031claim2_4}
Let $\bq:=(\xi^{\alpha})_{|\alpha|= 2}$.
Then
\begin{equation}\label{eqnclaim24}
\max_{m\ge 1}\max_{n'\le n}\frac{4n'}{2m+\numvar_{d-m, n'}(\bq)}=\frac{4n}{2+\numvar_{d-1, n}(\bq) }=\frac{4n}{d+1} = 2d.
\end{equation}
In other words, the $\max$ is attained at $m=1, n'=n$.
\end{claim}

\begin{proof}[Proof of Claim~\ref{201031claim2_4}]
By definition, $\numvar_{d-1, n} (\bq) = d-1$.
Hence it suffices to show the leftmost expression in (\ref{eqnclaim24}) is equal to $2d$.

Fix $m\ge 1$ and $n'\le n$.
Denote $l:=\numvar_{d-m, n'} (\bq)$.
Notice that by the definition of $\numvar_{d-m, n'} (\bq)$, we see that $l+m\le d$.
Our goal is to show that $\frac{4n'}{2m+l} \leq 2d$.
We claim that
\begin{equation}\label{201105e2_9}
n'\leq \binom{l+1}{2} + \binom{m+1}{2}+ m(d-m).
\end{equation}
Indeed, by definition \eqref{201105e1_15}, there exist a linear subspace $H \subset \R^{d}$ of dimension $d-m$ and a linear subspace $\calQ$ of the span of $\bq$ of dimension $n'$ such that the restrictions of the forms from $\calQ$ to $H$ depend only on $l$ variables.
Since the system $\bq=(\xi^{\alpha})_{|\alpha|=2}$ is a basis for the space of all quadratic forms in $d$ variables, the above statement does not depend on $H$ and the $l$ variables inside $H$, so we may assume $H=\Set{\xi: \xi_{d-m+1}=\dotsb=\xi_{d}=0}$ and the $l$ variables are $\xi_{1},\dotsc,\xi_{l}$.
In this case, $\calQ$ is contained in the space of all quadratic forms that depend either only on $\xi_{1},\dotsc,\xi_{l}$, or on at least one of the variables $\xi_{d-m+1},\dotsc,\xi_{d}$.
The right-hand side of \eqref{201105e2_9} is precisely the dimension of the latter space, which concludes the proof of \eqref{201105e2_9}.

Given \eqref{201105e2_9}, it remains to show
\begin{equation}
4(\frac{l(l+1)}{2} + \frac{m(m+1)}{2} + m(d-m)) \leq 2d(2m+l),
\end{equation}
which is equivalent to
\begin{equation}
2(l+m)(l-m+1) \leq 2dl.
\end{equation}
This holds because $l+m \leq d$ and $l-m+1 \leq l$.
\end{proof}

\section{Examples: Old and new}
\label{sec:examples}

\subsection{Example: Hypersurfaces with nonvanishing Gaussian curvatures}

We take $n=1$.
Let $Q$ be a quadratic form depending on $d$ variables.
Without loss of generality we assume that $\numvar_1(Q)=d$.
Via a change of coordinate, we can write $Q(\xi)$ as $\xi_1^2\pm\xi_2^2\pm\dots\pm\xi_d^2$.
This is the (hyperbolic) paraboloid case.
It is easy to see
$\numvar_{0}(Q)= 0, \numvar_{1}(Q) = d$.
\begin{lemma}\label{201021lem3.1}
Let $\widetilde{Q}: \R^d\to \R$ be a quadratic form.
Let $M\in M_{d\times d}$ with rank $d'$.
Then
\begin{equation}
\numvar_1(\widetilde{Q}(\cdot M))\ge \numvar_1(\widetilde{Q}(\cdot))-2(d-d').
\end{equation}
\end{lemma}
\begin{proof}[Proof of Lemma~\ref{201021lem3.1}]
A lemma of this form was already proved and used by Bourgain and Demeter, see Lemma 2.6 in \cite{MR3736493}.
We use $\mathrm{Hess}(\widetilde{Q})$ to denote the Hessian of the quadratic form $\widetilde{Q}$. What we need to prove is, for every $M\in M_{d\times d}$ with rank $d'$, it holds that
\begin{equation}\label{201026e3.2}
\rank (M\mathrm{Hess}(\widetilde{Q}) M^T)\ge \rank(\mathrm{Hess}(\widetilde{Q}))-2(d-d').
\end{equation}
This follows immediately form Sylvester's rank inequality:
\begin{equation}
\rank(AB)\ge \rank(A)+\rank(B)-n,
\end{equation}
for two arbitrary matrices $A, B\in M_{n\times n}$.
\end{proof}

By Lemma~\ref{201021lem3.1}, we know that $\numvar_1(Q|_H)\ge d-2m$ for every linear subspace of codimension $m$, which means $Q$ is non-degenerate.
Therefore we can apply Corollary~\ref{201017coro1.3} and obtain
\begin{equation}
\Gamma_p(Q)=\max\Bigl(d-\frac{2d+2}{p}, d(\frac 1 2-\frac 1 p)\Bigr).
\end{equation}
This recovers the $\ell^p L^p$ decoupling results of Bourgain and Demeter in \cite{MR3374964} and \cite{MR3736493}.
Moreover, if we take $Q(\xi)=\xi_1^2+\dots+\xi_d^2$, then it is elementary to see that $\numvar_1(Q|_H)\ge d-m$ for every linear subspace of co-dimension $m$, which means $Q$ is strongly non-degenerate.
Therefore we can apply Corollary~\ref{optimal-l2-decoupling} and obtain
\begin{equation}
\Gamma_{2, p}(Q)=\max\Bigl(0, \frac d 2-\frac{d+2}{p}\Bigr).
\end{equation}
This recovers the $\ell^2 L^p$ decoupling results of Bourgain and Demeter in \cite{MR3374964}.

\subsection{Example: Co-dimension two manifolds in \texorpdfstring{$\R^4$}{R4}}

Take $d=n=2$.
Let $Q_1(\xi)=A_1 \xi_1^2+2A_2 \xi_1 \xi_2+A_3\xi_2^2$ and $Q_2(\xi)=B_1 \xi_1^2+2B_2 \xi_1 \xi_2+B_3\xi_2^2$.
Under the assumption that
\begin{equation}\label{201017bd_non_dene}
\text{rank}
\begin{bmatrix}
A_1, & A_2, & A_3\\
B_1, & B_2, & B_3
\end{bmatrix}
=2,
\end{equation}
Bourgain and Demeter \cite{MR3447712} proved that
\begin{equation}
\Gamma_p(\bq)=\max\Bigl(2(\frac 1 2-\frac 1 p), 2(1-\frac 4 p)\Bigr),
\end{equation}
with $\bq=(Q_1, Q_2)$.
This decoupling inequality is particularly interesting as it is one key ingredient in Bourgain's improvement on the Lindel\"of Hypothesis in \cite{MR3556291}.

Let us see how Theorem~\ref{thm:main} recovers this result.
We take $d=n=2$ and notice that $\numvar_2(\bq)=2$ (indeed, if $\numvar_2(\bq)\leq 1$, then $Q_{1},Q_{2}$ would be linearly dependent, since the space of quadratic forms in one variable is one-dimensional, contradicting \eqref{201017bd_non_dene}).
Moreover, it is straightforward to see that $\numvar_1(\bq)>0$ as the assumption \eqref{201017bd_non_dene} says that $Q_1$ and $Q_2$ are linearly independent.
Therefore, $\bq$ is non-degenerate in the sense of \eqref{eq:non_degenerate}.
We can apply Corollary~\ref{201017coro1.3} and recover the result of Bourgain and Demeter \cite{MR3447712}.

\subsection{Example: Degenerate three-dimensional submanifolds of \texorpdfstring{$\R^5$}{R5}}\label{subsection_5_3}
Take $d=3$, $n=2$ and $\bq= (\xi_{1}^{2},\xi_{2}^{2}+\xi_{1}\xi_{3})$.
Note that $\numvar_{0}(\bq)=0$, $\numvar_{1}(\bq)=1$, $\numvar_{2}(\bq)=3$, and therefore $\bq$ fails to satisfy the non-degeneracy condition \eqref{eq:non_degenerate}.
On the other hand, one can also compute, for instance via \eqref{201105e1_15}, that $\numvar_{2, 2}(\bq)=1, \numvar_{2, 1}(\bq)=0$ and $\numvar_{d', n'}(\bq)=0$ whenever $d'\le 1$.
We apply Theorem~\ref{thm:main} and obtain that
\begin{equation}
\Gamma_p(\bq)=\max\Bigl(3(\frac 1 2-\frac 1 p), \frac 5 2-\frac 7 p, 3-\frac{10}{p} \Bigr),
\end{equation}
after some elementary computation.
This recovers the main result of Guo, Oh, Roos, Yung and Zorin-Kranich \cite{arxiv:1912.03995}, via an entirely different approach: The proof in \cite{arxiv:1912.03995} relies on bilinear Fourier restriction estimates, small cap decoupling inequalities for the parabola and the manifold $(\xi_1, \xi_2, \xi_1^2, \xi_1\xi_2)$ and a more sophisticated induction argument; while the proof in the current paper relies on more sophisticated Brascamp-Lieb inequalities and multi-linear Fourier restriction estimates.

\subsection{Simultaneously diagonalizable forms}

\begin{corollary}\label{201018coro3_2}
Let $\bq=(Q_1,\ldots,Q_n)$ be a collection of quadratic forms without mixed terms.
Then
\begin{equation}\label{1015.2.3}
\Gamma_p(\bq) =
\max_{0\leq n'\leq n}\Bigl( d \bigl(\frac{1}{2}-\frac{1}{p}\bigr) + \bigl(\frac{1}{2}-\frac{1}{p}\bigr) (d-\numvar_{n'}(\bq)) - \frac{2(n-n')}{p} \Bigr),
\end{equation}
for every $p\ge 2$.
\end{corollary}

\begin{proof}[Proof of Corollary~\ref{201018coro3_2}]
We first apply Corollary~\ref{1015.1.3} and obtain
\begin{equation}
\Gamma_p(\bq)=\max_{d/2 \leq d'\leq d} \max_{0\leq n'\leq n} \Big( (2d' - \numvar_{d',n'}(\bq)) \bigl(\frac{1}{2}-\frac{1}{p}\bigr) - \frac{2(n-n')}{p}\Big).
\end{equation}
In order to obtain \eqref{1015.2.3}, it suffices to prove that
\begin{equation}
\max_{d/2\le d'\le d} (2d'-\numvar_{d', n'}(\bq))=2d-\numvar_{n'}(\bq),
\end{equation}
for every $n'$.
By the equivalent definition of $\numvar_{d', n'}(\bq)$ as in \eqref{201105e1_15}, this is equivalent to proving
\begin{equation}
\min_{0 \leq m \leq d/2} \inf_{\substack{H \text{ of }\\Set{co\-dim }\ m}}
\bigl(\numvar_{n'}(\bq|_H)+2m\bigr)=\numvar_{n'}(\bq),
\end{equation}
for every $n'$, which is the same as saying
\begin{equation}\label{1015.2.5}
\numvar_{n'}(\bq)-2m \leq \numvar_{n'}(\bq|_H)
\end{equation}
for every $1 \leq n' \leq n$ and every plane $H$ of codimension $m$ with $1 \leq m \leq d/2$.

We argue by contradiction and assume that
\begin{equation}
\numvar_{n'}(\bq|_H)\le \numvar_{n'}(\bq)-2m-1,
\end{equation}
for some $n'$ and some linear subspace $H$ of codimension $m$.
By the definition \eqref{eq:numvar}, we can find $M_{d-m}\in GL_{d-m}(\R)$ and $M'\in M_{n\times n'}$ of rank $n'$ such that
\begin{equation}\label{201019e3_31}
\NV(\bar{\bp})=\numvar_{n'}(\bq|_H),
\end{equation}
where for $\xi'\in \R^{d-m}$ we define
\begin{equation}\label{201029e3_54}
\begin{split}
\bar{\bp}(\xi')&:=(Q_1|_H(\xi'\cdot M_{d-m}), \dots, Q_n|_H(\xi' \cdot M_{d-m}))\cdot M'\\
& = (Q_1((\xi' \cdot M_{d-m}, {\bf 0})\cdot \rot_H), \dots, Q_n((\xi' \cdot M_{d-m}, {\bf 0})\cdot \rot_H))\cdot M'.
\end{split}
\end{equation}
Here ${\bf 0}=(0, \dots, 0)\in \R^m$ and $\rot_H$ is a rotation matrix acting on $\R^d$.
Let $M_d\in GL_d(\R)$ be a matrix such that
\begin{equation}\label{201029e3_55}
(\xi' \cdot M_{d-m}, {\bf 0})=(\xi', {\bf 0})\cdot M_d, \text{ for every } \xi' \in \R^{d-m}.
\end{equation}
With this notation, we can write
\begin{equation}\label{201029e3_56}
\begin{split}
\bar{\bp}(\xi')&=(Q_1((\xi', {\bf 0})\cdot M_d\cdot \rot_H), \dots, Q_n((\xi', {\bf 0})\cdot M_d\cdot \rot_H))\cdot M'\\
&=:(\bar{P}_1(\xi'), \dots, \bar{P}_{n'}(\xi')).
\end{split}
\end{equation}
Recall \eqref{201019e3_31}.
It implies that
\begin{equation}\label{201019e3_36}
\NV(\lambda_1 \bar{P}_1+\dots+ \lambda_{n'} \bar{P}_{n'})
\le \numvar_{n'}(\bq|_H)
\le \numvar_{n'}(\bq)-2m-1,
\end{equation}
for all choices of $\lambda_1, \dots, \lambda_{n'}\in \R$.
Now if we denote
\begin{equation}
\bar{\bq}(\xi):=(\bar{Q}_1(\xi), \dots, \bar{Q}_{n'}(\xi)):=(Q_1(\xi), \dots, Q_n(\xi))\cdot M',
\end{equation}
then from the definition of $\numvar_{n'}(\bq)$ and the fact that $Q_1, \dots, Q_n$ are diagonal quadratic forms, we can find some $\lambda_1, \dots, \lambda_{n'}$ such that
\begin{equation}
\numvar_1(\lambda_1 \bar{Q}_1+\dots+\lambda_{n'} \bar{Q}_{n'})=
\NV(\lambda_1 \bar{Q}_1+\dots+\lambda_{n'} \bar{Q}_{n'})
\ge \numvar_{n'}(\bq).
\end{equation}
Recall the definition of $\bar{\bp}$ in \eqref{201029e3_54} and the relation in \eqref{201029e3_55} and \eqref{201029e3_56}.
Lemma~\ref{201021lem3.1} then says that
\begin{equation}
\NV(\lambda_1 \bar{P}_1+\dots+ \lambda_{n'} \bar{P}_{n'})
\ge \numvar_{n'}(\bq)-2m,
\end{equation}
which is a contradiction to \eqref{201019e3_36}.
\end{proof}

\begin{corollary}\label{1103.3.5}
For $1\le n'\le n$, define
\begin{equation}\label{201105e3_61}
Q_{n'}(\xi):=\sum_{1\le d'\le d} a_{n', d'} \xi_{d'}^2.
\end{equation}
Then for every $p\ge 2$, with $\bq=(Q_1, \dots, Q_{n})$,
\begin{equation}\label{362best}
\Gamma_p(\bq)=\max{\Big(d(\frac12-\frac1p),2d(\frac12-\frac1p)-\frac{2n}{p} \Big)}
\end{equation}
if and only if, for every $1 \leq n' \leq n$,
every $n\times (\floor{d-\frac{n'd}{n}}+1)$ submatrix of
\begin{equation}
\begin{bmatrix}\label{35assumption}
a_{1, 1}, & a_{1, 2}, & \dots, & a_{1, d}\\
\dots\\
a_{n, 1}, & a_{n, 2}, & \dots, & a_{n, d}
\end{bmatrix}
\end{equation}
has rank at least $n-n'+1$.
Here for $A\in \R$, $\floor{A}$ refers to the largest integer $\leq A$.\footnote{This notation is used only in Corollary~\ref{1103.3.5} and its proof.}
\end{corollary}
When $n=2$, a condition of the form \eqref{35assumption} already appeared in Heath-Brown and Pierce \cite{MR3652248}.
Let $\bq=(Q_1, Q_2)$ be a pair of quadratic forms with integer coefficients.
Heath-Brown and Pierce \cite{MR3652248} studied the problem of representing a pair of integers $(n_1, n_2)$ by the pair of $(Q_1, Q_2)$ for general $Q_1$ and $Q_2$.
If $Q_1$ and $Q_2$ are assumed to be simultaneously diagonalizable, say of the form \eqref{201105e3_61}, then the condition in \cite{MR3652248} becomes that every $2\times 2$ minor of
\begin{equation}
\begin{bmatrix}\label{35assumptionzz}
a_{1, 1}, & a_{1, 2}, & \dots, & a_{1, d}\\
a_{n, 1}, & a_{n, 2}, & \dots, & a_{n, d}
\end{bmatrix}
\end{equation}
has rank 2, see Condition 3 there.

\begin{proof}[Proof of Corollary~\ref{1103.3.5}]
Let us show the ``only if'' part by contradiction.
Suppose that, for some $1 \leq n' \leq n$, some $n \times (\floor{d-\frac{n'd}{n}}+1)$ submatrix of \eqref{35assumption} has rank $n-n'$ or less.
Then
\begin{equation}
\numvar_{n'}(\bq) \leq d-(\floor{d-\frac{n'd}{n}}+1)< \frac{n'd}{n}.
\end{equation}
Therefore, $\bq$ is not non-degenerate in the sense of~\ref{eq:non_degenerate}, and \eqref{362best} cannot hold true by Corollary~\ref{201017coro1.3}.

Let us show the other direction of the equivalence.
First of all, notice that the two terms on the right hand side of \eqref{362best} match at $p=p_{n, d}:=2+4n/d$.
By Corollary~\ref{201018coro3_2}, it suffices to show that
\begin{equation}
d \bigl(\frac{1}{2}-\frac{1}{p}\bigr) + \bigl(\frac{1}{2}-\frac{1}{p}\bigr) (d-\numvar_{n'}(\bq)) - \frac{2(n-n')}{p} \leq 2d(\frac12-\frac1p)-\frac{2n}{p}
\end{equation}
for every $1 \leq n' \leq n$ and every $p\ge p_{n, d}$.
By rearranging the terms, what we need to show becomes
\begin{equation}
\numvar_{n'}(\bq) \geq n'd/n
\end{equation}
for every $1 \leq n' \leq n$.
We argue by contradiction and assume that
\begin{equation}\label{366contradiction}
\numvar_{n'}(\bq) <n'd/n
\end{equation}
for some $1 \leq n' \leq n$.
By definition, there exist $M \in \R^{d \times d}$ of rank $d$ and $M' \in \R^{n \times n}$ of rank $n'$ such that
\begin{equation}
\numvar_{n'}(\bq)=
\NV(M' \cdot (\bq \circ M)).
\end{equation}
Since the assumption \eqref{35assumption} is invariant under the row operations, we may assume that $M'$ is a diagonal matrix with diagonal entries $1,\dotsc,1,0,\dotsc,0$.
By the inequality \eqref{366contradiction}, we have
\begin{equation}
\label{eq:7}
\dim \bigcap_{i=1}^{n'} \bigcap_{\xi\in\R^{d}} \ker \nabla Q_{i}(\xi) > d - \frac{n'd}{n}.
\end{equation}
It remains to observe that
\[
\ker \nabla Q_{i}(\xi) = \Set{ \eta\in\R^{d} \given \sum_{j=1}^{d}\xi_{j}\eta_{j}a_{i,j} = 0 }
\quad\text{and}
\]
\[
\bigcap_{\xi \in\R^{d}} \ker \nabla Q_{i}(\xi) = \Set{ \eta\in\R^{d} \given \eta_{j}a_{i,j} = 0 \text{ for all } j=1,\dotsc,d },
\]
so that \eqref{eq:7} implies that an $n' \times (\floor{d - \frac{n'd}{n}}+1)$ submatrix of \eqref{35assumption} vanishes.
\end{proof}

\subsection{Decoupling theory for two quadratic forms}
\begin{corollary}\label{201024coro3.5}
Let $\bq=(Q_1, Q_2)$ be two linearly independent quadratic forms defined on $\R^d$ satisfying $\numvar_2(\bq)=d$.
\begin{itemize}
\item[(1)] Let $1\le k< d/2$.
Then $\bq$ satisfies $\numvar_1(\bq)=k$ and the weakly non-degenerate condition if and only
\begin{equation}\label{20201028.376}
\Gamma_p(\bq)=\max\Bigl(d(\frac 1 2-\frac 1 p), (2d-k)(\frac 1 2-\frac 1 p)-\frac 2 p, 2d(\frac 1 2-\frac 1 p)-\frac 4 p\Bigr),
\end{equation}
for every $p\ge 2$.
\item[(2)] $\bq$ is non-degenerate if and only if it is weakly non-degenerate and satisfies $\numvar_1(\bq)\ge d/2$.
\end{itemize}
\end{corollary}
\begin{proof}[Proof of Corollary~\ref{201024coro3.5}]
Let us start with proving the first part of the corollary.
We denote the right hand side of \eqref{20201028.376} by $\Gamma'_p(\bq)$.
By Corollary~\ref{1015.1.3}, $\Gamma_p(\bq)$ is given by
\begin{equation}\label{201029e3_65}
\max_{d/2 \leq d' \leq d}\max\Bigl( (2d'-\numvar_{d',2}(\bq))(\frac 1 2-\frac 1 p), (2d'-\numvar_{d',1}(\bq))(\frac 1 2-\frac 1 p)-\frac 2 p, 2d(\frac 1 2-\frac 1 p)-\frac 4 p\Bigr).
\end{equation}
Let us first show that \eqref{20201028.376} holds, that is, $\Gamma_p(\bq)=\Gamma'_p(\bq)$ for every $p\ge 2$, if and only if
\begin{equation}\label{201029e3_66}
\begin{split}
&\max_{d'}(2d'-\numvar_{d',1}(\bq))=2d-k
\\&
\max_{d'}(2d'-\numvar_{d',2}(\bq))=d.
\end{split}
\end{equation}
To show that \eqref{201029e3_66} implies \eqref{20201028.376}, we apply \eqref{201029e3_65}, move the $\max_{d/2\le d'\le d}$ inside the second $\max$ and obtain \eqref{20201028.376}.
To show the other direction of the equivalence, the constraint $k<d/2$ will come into play.
Notice that under this assumption,
\begin{equation}
\Gamma'_p(\bq)=
\begin{cases}
d(\frac 1 2-\frac 1 p) & \text{ if } p\le 2+\frac{4}{d-k},\\
(2d-k)(\frac 1 2-\frac 1 p)-\frac 2 p & \text{ if } 2+\frac{4}{d-k}\le p\le 2+\frac{4}{k},\\
2d(\frac 1 2-\frac 1 p)-\frac 4 p & \text{ if } p\ge 2+\frac 4 k.
\end{cases}
\end{equation}
Note that we are now under the assumption that $\Gamma_p(\bq)=\Gamma'_p(\bq)$ for every $p\ge 2$.
When $p$ is slightly larger than 2, we have
\begin{equation}
\Gamma_p(\bq)=\max_{d/2\le d'\le d} (2d'-\numvar_{d', 2}(\bq))(\frac 1 2-\frac 1 p),
\end{equation}
as the contributions from the other two terms in \eqref{201029e3_65} are already negative.
This implies
\begin{equation}\label{201029e3_69}
\max_{d/2\le d'\le d}(2d'-\numvar_{d', 2}(\bq))=d.
\end{equation}
We use \eqref{201029e3_69} to further simplify $\Gamma_p(\bq)$ to
\begin{equation}
\max\Bigl( d(\frac 1 2-\frac 1 p), \max_{d/2 \leq d' \leq d} (2d'-\numvar_{d',1}(\bq))(\frac 1 2-\frac 1 p)-\frac 2 p, 2d(\frac 1 2-\frac 1 p)-\frac 4 p\Bigr).
\end{equation}
By comparing $\Gamma_p(\bq)$ with $\Gamma'_p(\bq)$ for $2+\frac{4}{d-k}\le p\le 2+\frac{4}{k}$, we see that
\begin{equation}
\max_{d/2\le d'\le d} (2d'-\numvar_{d', 1})=2d-k.
\end{equation}
This finishes the proof that \eqref{20201028.376} is equivalent to \eqref{201029e3_66}.

It remains to show that \eqref{201029e3_66} is equivalent to that $\bq$ is weakly non-degenerate and satisfies $\numvar_1(\bq)=k$.
Since the second equation in \eqref{201029e3_66} is already equivalent to the weakly non-degenerate condition, what we need to prove becomes $\numvar_{d, 1}(\bq)=\numvar_{1}(\bq)=k$ if and only if
\begin{equation}
\max_{d'}(2d'-\numvar_{d',1}(\bq))=2d-k,
\end{equation}
which follows immediately from
\begin{equation}\label{201029e3_73}
\max_{d'}(2d'-\numvar_{d',1}(\bq))=2d-\numvar_{d, 1}(\bq).
\end{equation}
To prove \eqref{201029e3_73}, it suffices to prove
\begin{claim}\label{201029claim3_7}
\begin{equation}
\numvar_{d,1}(\bq)-\numvar_{d',1}(\bq) \leq 2(d-d')
\end{equation}
for every $d/2 \leq d' \leq d$.
\end{claim}
The proof of Claim~\ref{201029claim3_7} will be presented in the end of this subsection.
So far we have finished the proof of the first part of the corollary.

Let us turn to the second part and show that $\bq$ is non-degenerate if and only if it is weakly non-degenerate and satisfies $\numvar_1(\bq) \geq d/2$.
By definition, we need to show that $\numvar_1(\bq) \geq d/2$ if and only if
\begin{equation}\label{20201028.380}
\numvar_{d-m,1}(\bq) \geq d/2-2m
\end{equation}
for every $0 \leq m \leq d/2$.
By taking $m=0$, we see that \eqref{20201028.380} implies $\numvar_1(\bq) \geq d/2$.
The other direction immediately follows from Claim~\ref{201029claim3_7}.
This finishes the second part of the corollary.
\end{proof}

\begin{proof}[Proof of Claim~\ref{201029claim3_7}]
We take $M_0 \in \R^{d \times d}$ of rank $d'$ and $M_0' \in \R^{2 \times 2}$ of rank one such that
\begin{equation}
\numvar_{d',1}(\bq)=
\inf_{\substack{M\in \R^{d\times d}\\ \rank(M)=d'}}
\inf_{\substack{M'\in \R^{n\times n}\\ \rank(M')=1}}\NV(M' \cdot (\bq \circ M))
=\NV(M_0' \cdot (\bq \circ M_0)).
\end{equation}
Therefore there exist $\lambda_1, \lambda_2\in \R$ such that
\begin{equation}
\numvar_{d',1}(\bq)
=\NV(\widetilde{Q}(\cdot M_0))
=\numvar_{d,1}(\widetilde{Q}(\cdot M_0))
\end{equation}
where $\widetilde{Q}=\lambda_1 Q_1+\lambda_2 Q_2$ and $\bq=(Q_1, Q_2)$.
We now apply Lemma~\ref{201021lem3.1} and obtain
\begin{equation}
\begin{split}
\numvar_{d',1}(\bq)-\numvar_{d,1}(\bq) & =\numvar_{d,1}(\widetilde{Q}(\cdot M_0))-\numvar_{d,1}(\bq)
\\&
\geq \numvar_{d,1}(\widetilde{Q}(\cdot M_0))-\numvar_{d,1}(\widetilde{Q})
\\&
\geq -2(d-d').
\end{split}
\end{equation}
This completes the proof of Claim~\ref{201029claim3_7}.
\end{proof}

\section{Transversality}
\label{sec:transversality}

\subsection{Brascamp--Lieb inequalities}
A central tool in most existing proofs of decoupling inequalities are the Brascamp--Lieb inequalities for products of functions in $\R^{m}$ which are constant along some linear subspaces.
Scale-invariant inequalities of this kind have been characterized in \cite{MR2377493}.
A novelty of our approach is that we for the first time take full advantage of scale-dependent versions of Brascamp--Lieb inequalities.
First inequalities of this kind were proved in \cite{MR2377493,MR2661170}, and a unified description taking into account both minimal and maximal scales was obtained in \cite{arxiv:1904.06450}.
We will only use the results of \cite{arxiv:1904.06450} in a special symmetric case when all functions $f_{j}$ below play similar roles.
This special case is captured in the following definition.
\begin{definition}
\label{def:BL}
Let $m,m' \in \N$.
Let $(V_{j})_{j=1}^{M}$ be a tuple of linear subspaces $V_{j} \subseteq \R^{m}$ of dimension $m'$.
For a linear subspace $V \subseteq \R^{m}$, let $\pi_{V}: \R^{m}\to V$ denote the orthogonal projection onto $V$.
For $0 \leq \alpha \leq M$ and $R\geq 1$, we denote by $\BL((V_{j})_{j=1}^{M},\alpha,R, \R^m)$ (for \emph{Brascamp--Lieb constant}) the smallest constant such that the inequality
\begin{equation}
\label{eq:BL}
\int_{[-R,R]^{m}} \avprod_{j=1}^{M} f_j(\pi_{V_{j}}(x))^{\alpha} \dif x
\leq \BL((V_{j})_{j=1}^{M},\alpha,R,\R^m)
\avprod_{j=1}^{M} \bigl( \int_{V_j} f_j(x_j) \dif x_j \bigr)^{\alpha}
\end{equation}
holds for any functions $f_{j} : V_{j} \to [0,\infty)$ that are constant at scale $1$, in the sense that $V_{j}$ can be partitioned into cubes with unit side length on each of which $f_{j}$ is constant.
If the dimension $m$ of the total space $\R^m$ is clear from the context, $\BL((V_{j})_{j=1}^{M},\alpha,R,\R^m)$ is often abbreviated to $\BL((V_{j})_{j=1}^{M},\alpha,R)$.
\end{definition}

We also need a Kakeya variant of Brascamp--Lieb inequalities, in which each function $f_{j}\circ\pi_{V_{j}}$ is replaced by a sum of functions of the form $f_{j,l}\circ\pi_{V_{j,l}}$, where $V_{j,l}$ are different subspaces.
The first almost optimal inequality of this kind was the multilinear Kakeya inequality, proved in \cite{MR2275834}, which generalizes the Loomis--Whitney inequality.
A simplified induction on scales proof was later given by Guth \cite{MR3300318}.
An endpoint version of the multilinear Kakeya inequality was proved by Guth \cite{MR2746348} using the polynomial method.
Endpoint Kakeya type extensions of Brascamp--Lieb inequalities were further developed in \cite{MR3019726,MR3738255,MR4129538}.
It will be convenient to use the following formulation, although a non-endpoint result such as \cite[Theorem 2]{arxiv:1904.06450} would also suffice for the purpose of proving decoupling inequalities with the optimal range of exponents.

\begin{theorem}[{Kakeya--Brascamp--Lieb, \cite{MR4129538}}]
\label{thm:KBL}
Fix integers $m' \leq m$.
Let $\calV_{j}$, $1\le j\le M$, be families of linear subspaces of $\R^{m}$ of dimension $m'$.
Let $1 \leq \alpha \leq M$ and $R\geq 1$.
Assume that
\begin{equation}
\label{eq:unif-bd-BL}
A := \sup_{V_{1} \in \calV_{1}, \dotsc, V_{M} \in \calV_{M}} \BL((V_{j})_{j=1}^{M},\alpha,R) < \infty.
\end{equation}
Then, for any non-negative integrable functions $f_{j,V_{j}} : V_{j} \to \R$ constant at scale $1$, we have
\begin{equation}
\label{eq:KBL}
\int_{B(0,R)} \avprod_{j=1}^M \Bigl( \sum_{V_{j} \in \calV_{j}} f_{j,V_{j}} (\pi_{V_{j}} (x)) \Bigr)^{\alpha} \dif x
\leq
C^{\alpha}A
\avprod_{j=1}^M \Bigl( \sum_{V_{j} \in \calV_{j}} \int_{V_{j}} f_{j,V_{j}} (x) \dif x \Bigr)^{\alpha},
\end{equation}
where the constant $C$ depends only on the dimension $m$.
\end{theorem}

The uniform bound \eqref{eq:unif-bd-BL} is clearly necessary for \eqref{eq:KBL} to hold.
In the scale invariant case, such uniform bounds for Brascamp-Lieb constants were obtained in \cite{MR3783217,MR3723636}.
We need the following corresponding result in the scale-dependent case.

\begin{theorem}[{\cite[Theorem 3]{arxiv:1904.06450}}]
\label{thm:BL}
In the situation of Definition~\ref{def:BL}, fix a tuple $(V_{j})_{j=1}^{M}$ and an exponent $1\leq \alpha \leq M$.
Let
\begin{equation}
\label{eq:kappa}
\kappa := \sup_{V \leq \R^{m}} \Bigl( \dim V - \frac{\alpha}{M} \sum_{j=1}^{M} \dim \pi_{V_{j}} V \Bigr),
\end{equation}
where the supremum is taken over all linear subspaces of $\R^{m}$.

Then there exists a constant $C_0<\infty$ and a neighborhood of the tuple $(V_{j})_{j=1}^{M}$ in the $M$-th power of the Grassmanian manifold of all linear subspaces of dimension $m'$ of $\R^{m}$ such that, for any tuple $(\tilde{V}_{j})_{j=1}^{M}$ in this neighborhood and any $R\geq 1$, we have
\begin{equation}
\label{eq:BL-uniform-bound}
\BL((\tilde{V}_{j})_{j=1}^{M},\alpha,R)
\leq
C_0 R^{\kappa}.
\end{equation}
\end{theorem}

\subsection{Transversality for quadratic forms}
Let $\bq=(Q_1, \dots, Q_n)$ be a sequence of quadratic forms defined on $\R^d$.
The subspaces in the subsequent application of Kakeya--Brascamp--Lieb inequalities will be the tangent spaces to the manifold $S_{\bq}$:
\begin{equation}
\label{eq:tangent-space}
V_{\xi}=V_{\xi}(\bq) := \lin \Set{ (e_{j}, \partial_{j}\bq(\xi)), j=1,\dots,d },
\quad \xi \in \R^{d}.
\end{equation}
Here $e_j$ is the $j$-th coordinate vector and $\lin$ refers to linear span.
Transversality of pieces of this manifold will be measured by the exponent $\kappa$ defined in \eqref{eq:kappa} evaluated at tangent spaces somewhere at the respective pieces: The smaller the exponent, the more transverse are the pieces.
It is an observation going back to \cite{MR3614930} (for scale-invariant Brascamp--Lieb inequalities) that the most transverse situations arise when the pieces are not concentrated near a low degree subvariety in the following sense.

\begin{definition}
\label{def:uniform}
A subset $\calW \subseteq \Part{1/K}$ will be called \emph{$\theta$-uniform} if, for every non-zero polynomial $P$ in $d$ variables with real coefficients of degree $\leq d$, we have
\[
\abs{\Set{W \in \calW \given 2W \cap Z_{P}\neq \emptyset}} \leq \theta \abs{\calW}.
\]
Here $Z_P$ refers to the zero set of $P$.
When using the notation $\calW = \Set{W_{1},\dotsc, W_{M}} = \Set{ W_{j} }_{j=1}^{M}$ for $\theta$-uniform sets, we always mean that the $W_{j}$'s are pairwise distinct.
\end{definition}

\begin{lemma}
\label{lem:BL-uniform}
Let $\theta \in [0,1]$, $\alpha \geq 1$, and $K \in 2^{\N}$.
Then there exists $C_{\theta,K,\alpha} < \infty$ such that, for every $\theta$-uniform set $\calW = \Set{W_{1},\dotsc, W_{M}} \subseteq \Part{1/K}$ with $\alpha\leq M$ and every $R\geq 1$, we have
\[
\sup_{\xi_{j}\in W_{j}} \BL((V_{\xi_{j}})_{j=1}^{M},\alpha,R, \R^{d+n})
\leq
C_{\theta,K,\alpha} R^{\kappa_{\bq}(\alpha\cdot(1-\theta))},
\]
where
\begin{equation}
\label{eq:kappa(alpha)}
\kappa_{\bq}(\alpha)
:=
\sup_{V \leq \R^{d+n}} \Big(\dim V - \alpha \sup_{\xi \in \R^{d}}\dim \pi_{V_{\xi}} V\Big).
\end{equation}
\end{lemma}
In the remaining part, if $\bq$ is clear from the context, we often abbreviate $\kappa_{\bq}(\alpha)$ to $\kappa(\alpha)$.
\begin{proof}[Proof of Lemma~\ref{lem:BL-uniform}]
Since there are only finitely many $\theta$-uniform sets $\calW$, and, for any fixed $\theta$-uniform set $\calW = \Set{W_{1},\dotsc, W_{M}}$, the set $\prod_{j=1}^{M} \overline{W_j}$ is compact, by Theorem~\ref{thm:BL}, it suffices to show that, for any $\xi_{j}\in \overline{W_{j}}$, and every subspace $V \leq \R^{d+n}$, we have
\[
\dim V - \frac{\alpha}{M} \sum_{j=1}^{M} \dim \pi_{V_{\xi_{j}}} V
\leq
\dim V - \alpha(1-\theta) \sup_{\xi \in \R^{d}}\dim \pi_{V_{\xi}} V.
\]
This is equivalent to
\[
\frac{1}{M} \sum_{j=1}^{M} \dim \pi_{V_{\xi_{j}}} V
\geq
(1-\theta) \sup_{\xi \in \R^{d}}\dim \pi_{V_{\xi}} V.
\]
If $v_{1},\dotsc,v_{m}$ is a basis of $V$, then
\begin{equation}
\label{eq:dim-pi-V-as-rank}
\dim \pi_{V_{\xi}} V
=
\rank
\begin{pmatrix}
e_{1} & \partial_{1} \bfQ(\xi)\\
\vdots & \vdots\\
e_{d} & \partial_{d} \bfQ(\xi)\\
\end{pmatrix}
\cdot
\begin{pmatrix}
v_{1} & \dots & v_{m}
\end{pmatrix},
\end{equation}
where on the right hand side we have the product of two matrices.
Each minor determinant of this matrix is a polynomial of degree at most $d$.
Consider the largest minor (of size $d'\times d'$, say) whose determinant is a non-vanishing polynomial; call this polynomial $P$ (if $d'=0$, then $P=1$).
Then
\[
d' = \sup_{\xi \in \R^{d}}\dim \pi_{V_{\xi}} V.
\]
By Definition~\ref{def:uniform}, we have $P(\xi_{j})\neq 0$ for at least $(1-\theta)M$ many $j$'s.
Therefore,
\[
\frac{1}{M} \sum_{j=1}^{M} \dim \pi_{V_{\xi_{j}}} V
\geq
\frac{1}{M} \sum_{j : P(\xi_{j})\neq 0} \dim \pi_{V_{\xi_{j}}} V
\geq
\frac{1}{M} \sum_{j : P(\xi_{j})\neq 0} d'
\geq
(1-\theta) d'.
\qedhere
\]
This finishes the proof of the lemma.
\end{proof}

From the proof of Lemma~\ref{lem:BL-uniform}, we see that the $\sup$ in $\sup_{\xi\in \R^d}\dim \pi_{V_\xi} V$ is attained at almost every point, with respect to the $d$-dimensional Lebesgue measure.
Therefore, we introduce the following notation
\begin{equation}
\dim \pi V:=\sup_{\xi\in \R^d}\dim \pi_{V_\xi} V.
\end{equation}
Next, we will find a more explicit description of the exponent \eqref{eq:kappa(alpha)} in terms of the quadratic forms $\bq$.
The following result relates the terms in \eqref{eq:kappa(alpha)} to the quantities introduced in \eqref{eq:numvar}.

\begin{lemma}
\label{lem:dim-proj-vs-missing-variables}
Let $\bq$ be an $n$-tuple of quadratic forms in $d$ variables.
For a linear subspace $V \subseteq \R^{d+n}$, let
\[
d' := \dim \pi V,
\quad
n' := \dim V - \dim \pi V.
\]
Then
\[
n' \leq n
\text{ and }
\numvar_{n'}(\bq) \leq d'.
\]
\end{lemma}
Lemma~\ref{lem:dim-proj-vs-missing-variables} relies on the following algebraic result.
\begin{lemma}
\label{lem:zero-block}
Let $\mathbb{F} = \R(\xi_{1},\dotsc,\xi_{d})$ be the field of rational functions in $d$ variables.
Let $A = (\sum_{k}a_{i,j,k}\xi_{k})_{i,j}$ be a $(N_{1}\times N_{2})$-matrix whose entries are linear maps with real coefficients.
Suppose that $\rank_{\mathbb{F}} A = r$.
Then there exist real invertible matrices $B,B'$ such that
\begin{equation}
\label{eq:zero-block}
BAB' =
\begin{pmatrix}
* & *\\
* & 0
\end{pmatrix},
\end{equation}
where the zero block has size $(N_{1}-r)\times (N_{2}-r)$.
\end{lemma}
Standard linear algebra shows that there exist invertible matrices $B,B'$ with entries in $\mathbb{F}$ such that \eqref{lem:zero-block} holds.
The point of this Lemma~\ref{lem:zero-block} is that we can find $B,B'$ with real entries.
Note that Lemma~\ref{lem:zero-block} may fail if entries of $A$ are not assumed to be linear forms. We will include a proof of Lemma~\ref{lem:zero-block} below and we also note that after finishing the first version of the paper, Zipei Nie \cite{arxiv:2111.05553} pointed out to us that Lemma~\ref{lem:zero-block} is in fact known in the literature and follows from \cite[Lemma 1]{MR136618}. We thank him for this comment.

\begin{proof}[Proof of Lemma~\ref{lem:dim-proj-vs-missing-variables} assuming Lemma~\ref{lem:zero-block}]
The claim $n'\leq n$ follows from the fact that the tangent spaces $V_{\xi}$ have codimension $n$.

After linear changes of variables in $\R^{d}$ and $\R^{n}$, we may assume that $V$ is spanned by linearly independent vectors of the form
\[
(e_{1},v_{1}),\dotsc,(e_{s},v_{s}),(0,\tilde{e}_{1}),\dotsc,(0,\tilde{e}_{j}),
\]
where $e_{i}$ are unit coordinate vectors in $\R^{d}$, $\tilde{e}_{i}$ are unit coordinate vectors in $\R^{n}$, $v_i$ are vectors in $\R^n$ and $s \leq \min(d,\dim V)$.
Note also that $\dim V=s+j$ and that $j\le n$.
As in \eqref{eq:dim-pi-V-as-rank}, we have
\[
\dim \pi_{V_\xi} V
=
\rank_{\R}
\begin{pmatrix}
e_{1} & \partial_{1} \bfQ(\xi)\\
\vdots & \vdots\\
e_{d} & \partial_{d} \bfQ(\xi)\\
\end{pmatrix}
\cdot
\begin{pmatrix}
e_{1}^{T} & \dotsc & e_{s}^{T} & 0 & \dotsc & 0\\
v_{1}^{T} & \dotsc & v_{s}^{T} & \tilde{e}_{1}^{T} & \dotsc & \tilde{e}_{j}^{T}\\
\end{pmatrix}.
\]
Since all entries of the product matrix on the right-hand side are polynomials in $\xi$, we have
\[
d'
=
\sup_{\xi} \dim \pi_{V_\xi} V
=
\rank_{\mathbb{F}}
\begin{pmatrix}
e_{1} & \partial_{1} \bfQ(\xi)\\
\vdots & \vdots\\
e_{d} & \partial_{d} \bfQ(\xi)\\
\end{pmatrix}
\cdot
\begin{pmatrix}
e_{1}^{T} & \dotsc & e_{s}^{T} & 0 & \dotsc & 0\\
v_{1}^{T} & \dotsc & v_{s}^{T} & \tilde{e}_{1}^{T} & \dotsc & \tilde{e}_{j}^{T}\\
\end{pmatrix},
\]
where $\mathbb{F}$ is the field of rational functions in $d$ variables.
This is because the rank equals the size of the largest minor with non-vanishing determinant, and the determinant of any minor, viewed as an element of $\mathbb{F}$, vanishes if and only if its value vanishes for every $\xi$.
The latter matrix can be written in the block form
\begin{equation}
\label{eq:6}
\begin{pmatrix}
I + L_{1} & L_{3}\\
L_{2} & B
\end{pmatrix},
\end{equation}
where $I$ is the $s\times s$ identity matrix, $L_{1},L_{2},L_{3}$ are matrices whose entries are linear combinations of monomials of degree $1$, and
\[
B
=
\begin{pmatrix}
\partial_{s+1} Q_{1}(\xi) & \dotsc & \partial_{s+1} Q_{j}(\xi)\\
\vdots & & \vdots \\
\partial_{d} Q_{1}(\xi) & \dotsc & \partial_{d} Q_{j}(\xi)
\end{pmatrix}.
\]
Let $r:=\rank_{\mathbb{F}} B$.
Any $r\times r$-minor determinant $P$ of the matrix $B$ is a homogeneous polynomial of degree $r$, and $P$ coincides with the lowest degree homogeneous part of the corresponding $(r+s)\times(r+s)$-minor determinant of \eqref{eq:6}, obtained by adjoining the first $s$ rows and columns. Therefore,
\[
d' \geq s + r.
\]

Let us continue to prove $\numvar_{n'}(\bq)\le d'$.
Recall that $n' = s+j-d' \leq j-r$.
By the definition in \eqref{eq:numvar}, it suffices to find linear changes of variables in $\R^{d}$ and $\R^{n}$, after which $Q_{r+1},\dotsc,Q_{j}$ no longer depend on variables $\xi_{s+r+1},\dotsc,\xi_{d}$.
Notice that row and column operations on $B$ with coefficients in $\R$ correspond to linear changes of variables in $\R^{d}$ and $\R^{n}$, respectively.
By Lemma~\ref{lem:zero-block}, by row and column operations with coefficients in $\R$, $B$ can be brought in a form in which it has a $(j-r) \times (d-s-r)$-block of zeroes.
This means that, after a change of variables, $Q_{r+1},\dotsc,Q_{j}$ do not depend on variables $\xi_{s+r+1},\dotsc,\xi_{d}$.
\end{proof}

\begin{proof}[Proof of Lemma~\ref{lem:zero-block}]
Let $k_{1}$ be the largest index such that $\xi_{k_{1}}$ appears in $A$.
Swapping rows and columns, we may assume $a_{1,1,k_{1}}\neq 0$.
Using elementary row and column operations, we may further assume that $a_{1,1,k_{1}}=1$, $a_{1,j,k_{1}}=0$, and $a_{i,1,k_{1}}=0$ for all $j\neq 0$ and $i\neq 0$.
Thus, we may assume
\[
A =
\begin{pmatrix}
\xi_{k_{1}} + * & *\\
* & A'
\end{pmatrix},
\]
where $\xi_{k_{1}}$ does not appear in entries $*$ and $A'$ is an $(N_{1}-1)\times (N_{2}-1)$-matrix.
If $A' \neq 0$, we repeat the same procedure in $A'$, and so on.
If this process stops after at most $r$ iterations, then we are done.
Otherwise, we have brought the upper left corner of $A$ into the form
\begin{equation}
\label{eq:3}
\begin{pmatrix}
\xi_{k_{1}} + * & * & \dots & *\\
* & \xi_{k_{2}} + * & \dots & *\\
\vdots & & \ddots & \vdots\\
* & \dots & * & \xi_{k_{r+1}} + *
\end{pmatrix},
\end{equation}
where $a_{i,j,k}=0$ if $i\neq j$ and $k \geq k_{\min(i,j)}$.
The determinant of this matrix is a polynomial whose leading term in the lexicographic ordering is $\xi_{k_{1}}\dotsm \xi_{k_{r+1}}$:
\[
\det \eqref{eq:3}
=
\xi_{k_{1}}\dotsm \xi_{k_{r+1}} + \text{lower order terms.}
\]
This can be seen by induction on the size of this matrix.
Indeed, if $k_{1}=\dotsb=k_{l} > k_{l+1}$, then $\xi_{k_{1}}$ appears in this matrix only in the first $l$ diagonal entries, so
\[
\det \eqref{eq:3}
=
\xi_{k_{1}}^{l}
\cdot \det
\begin{pmatrix}
\xi_{k_{l+1}} + * & * & \dots & *\\
* & \xi_{k_{l+2}} + * & \dots & *\\
\vdots & & \ddots & \vdots\\
* & \dots & * & \xi_{k_{r+1}} + *
\end{pmatrix}
+\text{lower order terms.}
\]
In particular, the matrix \eqref{eq:3} is invertible (over $\mathbb{F}$), so that $\rank_{\mathbb{F}}A \geq r+1$, a contradiction.
\end{proof}

\begin{corollary}
\label{cor:kappa-bound}
For any $\alpha\geq 1$, the exponent defined in \eqref{eq:kappa(alpha)} satisfies
\begin{equation}
\label{eq:kappa-bound}
\kappa(\alpha)
\leq
\sup_{0 \leq n' \leq n} \Big(n' + (1-\alpha) \numvar_{n'}(\bq)\Big).
\end{equation}
\end{corollary}
\begin{proof}[Proof of Corollary~\ref{cor:kappa-bound}]
Let $V \subseteq \R^{d+n}$ be a linear subspace.
With the notation from Lemma~\ref{lem:dim-proj-vs-missing-variables}, we obtain
\[
\dim V - \alpha \dim \pi V
=
(d' + n') - \alpha d'
=
n' + (1-\alpha) d'
\leq
n' + (1-\alpha) \numvar_{n'}(\bq).
\]
The conclusion follows after taking the supremum over all subspaces $V$.
\end{proof}

\subsection{Ball inflation}
A so-called ball inflation inequality, based on scale invariant Kakeya--Brascamp--Lieb inequalities, was first introduced in \cite[Theorem 6.6]{MR3548534}.
Here, we formulate a version of this inequality based on scale-dependent Kakeya-Brascamp-Lieb inequalities.
Recall that $\calU_{J}$ was defined in \eqref{eq:Uncertainty-box}.

\begin{proposition}[Ball inflation]\label{prop:ball-inflation}
Let $K \in 2^{\N}$ be a dyadic integer and $0 < \rho \leq 1/K$.
Let $\Set{W_j}_{j=1}^{M} \subset \Part{1/K}$ be a $\theta$-uniform set of cubes.
Then, for any $1 \leq t \leq p < \infty$, any functions $f_{J}$ with $\supp \widehat{f_{J}} \subset \calU_{J}$ and any $x_{0} \in \R^{d+n}$, we have
\begin{equation}
\label{eq:ball-inflation}
\begin{split}
& \norm[\Big]{\avprod_{j=1}^{M} \Big(\sum_{J \in \Part[W_j]{\rho}}\norm{f_{J}}_{\avL^{t}(w_{B(x,1/\rho)})}^t\Big)^{1/t}}_{\avL^{p}_{x \in B(x_{0},\rho^{-2})}}\\
& \le C_{\theta,K,p,t} \rho^{-(\frac{d}{t} - \frac{d+n}{p} + \frac{\kappa((1-\theta)p/t)}{p})} \avprod_{j=1}^{M} \Big(\sum_{J \in \Part[W_j]{\rho}}\norm{f_{J}}_{\avL^{t}(w_{B(x_{0},\rho^{-2})})}^t \Big)^{1/t}.
\end{split}
\end{equation}
\end{proposition}
\begin{proof}[Proof of Proposition~\ref{prop:ball-inflation}]
Let $R:=\rho^{-1}$.
Without loss of generality, we set $x_{0}=0$.
Let $\Omega := B(0,R^{2})$.
Let $\alpha := p/t$.
The $p$-th power of the left-hand side of \eqref{eq:ball-inflation} equals
\begin{equation}
\label{eq:ball-inflation-lhs-power-p}
\fint_{x\in \Omega} \avprod_{j=1}^M \Bigl( \sum_{J \in \Part[W_j]{\rho}} \norm{ f_{J} }_{\avL^{t}(w_{B(x,1/\rho)})}^t \Bigr)^{\alpha},
\end{equation}
where $\fint_{\Omega} :=|\Omega|^{-1}\int_{\Omega}$ denotes the average integral.
For each cube $J \in \Part[W_j]{\rho}$ with center $\xi_{J}$, we cover $\Omega$ with a family $\calT_{J}$ of disjoint tiles $T_{J}$, which are rectangular boxes with $n$ long sides of length $2\rho^{-2}$ centered at $0$ pointing in the directions $V_{\xi_{J}}^{\perp}$ and $d$ short sides of length $\rho^{-1}$ pointing in complementary directions (the length of the long sides equals the diameter of $B(x_{0},\rho^{-2})$, so that we only need one layer of tiles in the directions $V_{\xi_{J}}^{\perp}$).
We can choose these tiles so that they are contained in $C_{0}\Omega$ with $C_{0}\lesssim 1$.
We let $T_{J}(x)$ be the tile containing $x$, and for $x\in \cup_{ T_{J}\in\calT_{J}} T_{J}$ we define
\[
F_{J}(x) := \sup_{y \in T_{J}(x)} \norm{ f_{J} }_{\avL^{t}(w_{B(y, 1/\rho)})}.
\]
Then
\[
\fint_{x\in \Omega} \avprod_{j=1}^M \Bigl( \sum_{J \in \Part[W_j]{\rho}} \norm{ f_{J} }_{\avL^{t}(w_{B(x,1/\rho)})}^t \Bigr)^{\alpha}
\leq
\fint_{\Omega} \avprod_{j=1}^M \bigl( \sum_{J \in \Part[W_j]{\rho}} |F_{J}|^t \bigr)^{\alpha}.
\]
Since the function $F_{J}$ is constant on each tile $T_{J}\in\calT_{J}$, we can write its restriction to $\Omega$ in the form $\tilde{F}_{J} \circ \pi_{J}$, where $\pi_{J}$ is the orthogonal projection onto $V_{\xi_{J}}$.
To apply Theorem~\ref{thm:KBL}, we apply the change of variables $x\to R x$ such that the resulting functions $F_{J}(Rx)$ are constant at the unit scale:
\begin{equation}
\fint_{\Omega} \avprod_{j=1}^M \bigl( \sum_{J \in \Part[W_j]{\rho}} |F_{J}|^t \bigr)^{\alpha}=
R^{-(d+n)} \int_{B(0,R)} \avprod_{j=1}^M \bigl( \sum_{J \in \Part[W_j]{\rho}} F_{J}(R\cdot)^t \bigr)^{\alpha}.
\end{equation}
By Theorem~\ref{thm:KBL} and Lemma~\ref{lem:BL-uniform} with $R=\rho^{-1}$, we bound the last expression by
\begin{equation}
\begin{split}
& R^{-(d+n)} C_{\theta,K,\alpha} R^{\kappa(\alpha(1-\theta))}
\avprod_{j=1}^M \Bigl( \sum_{J \in \Part[W_j]{\rho}} \int_{B(0,C_{0}R) \subset \R^{d}} (\tilde{F}_{J}(R\cdot))^t \Bigr)^{\alpha}\\
& \lesssim R^{-(d+n)} C_{\theta,K,\alpha} R^{\kappa(\alpha(1-\theta))} R^{\alpha d}
\avprod_{j=1}^M \Bigl( \sum_{J \in \Part[W_j]{\rho}} \fint_{B(0,C_{0}R^{2}) \subset \R^{d+n}} |F_{J}|^t \Bigr)^{\alpha}.
\end{split}
\end{equation}
The conclusion will now follow from the bound
\begin{equation}
\label{FJbound}
\norm{ F_{J} }_{\avL^{t}(C_{0}\Omega)}
\lesssim
\norm{ f_{J} }_{\avL^{t}(w_{\Omega})},
\end{equation}
which is a standard application of the uncertainty principle, see e.g.\ \cite[(3.13)]{MR4031117}.
\end{proof}

The ball inflation inequality in Proposition~\ref{prop:ball-inflation} is sufficient for proving $\ell^{p}L^{p}$ decoupling.
For the proof of $\ell^{q}L^{p}$ decoupling with $q<p$ we need a slightly more general statement.
\begin{corollary}[Ball inflation, $\ell^{q}L^{t}$ version]\label{cor:ball-inflation:lqLt}
In the situation of Proposition~\ref{prop:ball-inflation}, for any $1 \leq \tilde{q} \leq t$, we have
\begin{equation}
\label{eq:ball-inflation:lqLt}
\begin{split}
& \norm[\Big]{ \avprod_{j=1}^{M} \Big(\sum_{J \in \Part[W_j]{\rho}} \norm{f_{J}}_{\avL^{t}(w_{B(x,1/\rho)})}^{\tilde{q}}\Big)^{1/\tilde{q}}}_{\avL^{p}_{x \in B(x_{0},\rho^{-2})}}\\
& \leq C_{\theta,K,p,t, \tilde{q}} (\abs{\log\rho}+2)^{K^{d}} \rho^{-(\frac{d}{t} - \frac{d+n}{p} + \frac{\kappa((1-\theta)p/t)}{p})}
\avprod_{j=1}^{M} \Big( \sum_{J \in \Part[W_j]{\rho}}\norm{f_{J}}_{\avL^{t}(w_{B(x_{0},\rho^{-2})})}^{\tilde{q}}\Big)^{1/\tilde{q}}.
\end{split}
\end{equation}
\end{corollary}
\begin{proof}[Proof of Corollary~\ref{cor:ball-inflation:lqLt}]
This follows from Proposition~\ref{prop:ball-inflation} by a dyadic pigeonholing argument in the proof of \cite[Theorem 6.6]{MR3548534}.
For the sake of completeness, we still include the proof here.
We follow the presentation in \cite[Appendix A]{MR4031117}.

For each $1\le j\le M$, partition
\begin{equation}
\calP(W_j, \rho)=\calJ_{j, \infty} \cup \bigcup_{\iota=0}^{\floor{d\log_{2}(1/\rho)}}\calJ_{j, \iota},
\end{equation}
where for $0\le \iota\le \log(1/\rho)$
\begin{equation}
\begin{split}
\calJ_{j, \iota}
&:=\Set[\Big]{J\in \calP(W_j, \rho): 2^{-\iota-1}<\frac{\norm{f_J}_{\avL^{t}(w_B)}}{\max_{J'\in \calP(W_j, \rho)}\norm{f_{J'}}_{\avL^{t}(w_B)}}\le 2^{-\iota}},
\\
\calJ_{j, \infty}
&:=\Set[\Big]{J\in \calP(W_j, \rho): {\norm{f_J}_{\avL^{t}(w_B)}}\le 2^{-\floor{d\log_{2}{(1/\rho)}}}{\max_{J'\in \calP(W_j, \rho)}\norm{f_{J'}}_{\avL^{t}(w_B)}} }.
\end{split}
\end{equation}
Since $M \leq K^{d}$, the claim \eqref{eq:ball-inflation:lqLt} follows by the triangle inequality from
\begin{equation}
\label{eq:ball-inflation:lqLt_cc}
\begin{split}
& \norm[\Big]{ \avprod_{j=1}^{M} \Big(\sum_{J \in \calJ_{j, \iota_j}} \norm{f_{J}}_{\avL^{t}(w_{B(x,1/\rho)})}^{\tilde{q}}\Big)^{1/\tilde{q}}}_{\avL^{p}_{x \in \Omega}}\\
& \leq C_{\theta,K,p,t, \tilde{q}} \rho^{-(\frac{d}{t} - \frac{d+n}{p} + \frac{\kappa((1-\theta)p/t)}{p})}
\avprod_{j=1}^{M} \Big( \sum_{J \in \calP(W_j, \rho)}\norm{f_{J}}_{\avL^{t}(w_\Omega)}^{\tilde{q}}\Big)^{1/\tilde{q}},
\end{split}
\end{equation}
which we will show for every choice of $\iota_1, \dotsc, \iota_M\in \Set{0, \dotsc, \floor{d\log_{2}(1/\rho)}} \cup \Set{\infty}$.

Since $\tilde{q}\le t$, by H\"older's inequality, the left hand side of \eqref{eq:ball-inflation:lqLt_cc} is bounded by
\begin{equation}
\Big(\avprod |\calJ_{j, \iota_j}|^{\frac{1}{\tilde{q}}-\frac{1}{t}}\Big) \norm[\Big]{ \avprod_{j=1}^{M} \Big(\sum_{J \in \calJ_{j, \iota_j}} \norm{f_{J}}_{\avL^{t}(w_{B(x,1/\rho)})}^{t}\Big)^{1/t}}_{\avL^{p}_{x \in B}}.
\end{equation}
By Proposition~\ref{prop:ball-inflation}, the last display is bounded by
\begin{equation}
C_{\theta,K,p,t, \tilde{q}} \rho^{-(\frac{d}{t} - \frac{d+n}{p} + \frac{\kappa((1-\theta)p/t)}{p})}
\Big(\avprod |\calJ_{j, \iota_j}|^{\frac{1}{\tilde{q}}-\frac{1}{t}}\Big) \avprod_{j=1}^{M} \Big( \sum_{J \in \calJ_{j, \iota_j}}\norm{f_{J}}_{\avL^{t}(w_B)}^{t}\Big)^{1/t}.
\end{equation}
It remains to observe that, for every $\iota$, we have
\begin{equation}\label{201031e4_25}
|\calJ_{j, \iota}|^{\frac{1}{\tilde{q}}-\frac{1}{t}}\Big( \sum_{J \in \calJ_{j, \iota}}\norm{f_{J}}_{\avL^{t}(w_B)}^{t}\Big)^{1/t} \lesssim
\Big( \sum_{J \in \calP(W_j, \rho)} \norm{f_{J}}_{\avL^{t}(w_B)}^{\tilde{q}}\Big)^{1/\tilde{q}}.
\end{equation}
If $\iota\neq \infty$, this follows as the summands on the left hand side are comparable.
For $\iota=\infty$, we have
\begin{equation}
\begin{split}
\MoveEqLeft
|\calJ_{j, \infty}|^{\frac{1}{\tilde{q}}-\frac{1}{t}}
\Big( \sum_{J \in \calJ_{j, \infty}} \norm{f_{J}}_{\avL^{t}(w_B)}^{t}\Big)^{1/t}
\\ &\leq
|\calJ_{j, \infty}|^{\frac{1}{\tilde{q}}}
\max_{J \in \calJ_{j, \infty}} \norm{f_{J}}_{\avL^{t}(w_B)}
\\ &\leq
\rho^{-d} 2^{-\floor{d\log_{2}(1/\rho)}} \max_{J'\in \calP(W_j, \rho)} \norm{f_{J'}}_{\avL^{t}(w_B)}
\\ &\lesssim
\Big( \sum_{J \in \calP(W_j, \rho)} \norm{f_{J}}_{\avL^{t}(w_B)}^{\tilde{q}}\Big)^{1/\tilde{q}}.
\qedhere
\end{split}
\end{equation}
\end{proof}

\section{Induction on scales}
\label{sec:induction-on-scales}
The upper bounds of $\Gamma_{q, p}(\bq)$ in Theorem~\ref{thm:main} will be proved by induction on the dimension $d$.
The main inductive step is contained in the following result, whose proof will occupy the whole Section~\ref{sec:induction-on-scales}.
One can apply Theorem~\ref{thm:main-inductive-step} repeatedly and then obtain the upper bounds in part \eqref{eq:dec-exponent-unrolled} of Theorem~\ref{thm:main}.
\begin{theorem}
\label{thm:main-inductive-step}
Let $d\ge 1$ and $n\ge 1$.
Let $\bq=(Q_1, \dotsc, Q_n)$ be a collection of quadratic forms in $d$ variables.
Let $2 \leq q \leq p < \infty$ and
\begin{equation}
\label{eq:lower-dim-dec-exponent}
\Lambda := \sup_{H}\Gamma_{q,p}(\bq|_H),
\end{equation}
where the sup is taken over all hyperplanes $H \subset \R^d$ that pass through the origin.
Then
\begin{equation}
\label{eq:dec-exponent:inductive-step}
\Gamma_{q,p}(\bq) \leq
\max \Bigl( \Lambda , \max_{0\leq n'\leq n} \Big[ d\bigl(1-\frac{1}{p}-\frac{1}{q} \bigr) - \numvar_{n'}(\bq)\bigl(\frac{1}{2}-\frac{1}{p}\bigr) - \frac{2(n-n')}{p}\Big] \Bigr).
\end{equation}
\end{theorem}
In the proof of Theorem~\ref{thm:main-inductive-step}, we may assume that
\begin{equation}
\label{eq:Gamma>Lambda}
\Gamma := \Gamma_{q,p}(\bq) > \Lambda,
\end{equation}
since otherwise \eqref{eq:dec-exponent:inductive-step} already holds.
The assumption \eqref{eq:Gamma>Lambda} is convenient, because it means that the multilinear terms in Proposition~\ref{prop:Bourgain-Guth} below dominate the lower-dimensional terms.
On a technical level, it allows us to define the quantities \eqref{eq:tilde-a} that are central to the bootstrapping argument.

\subsection{Stability of decoupling constants and lower dimensional contributions}
In order to make use of the quantity \eqref{eq:lower-dim-dec-exponent}, we need to show that we have a bound for the decoupling constants $\Dec_{q,p}(\bq|_{H},\delta)$ that is uniform in the hyperplanes $H$.
More generally, it turns out that decoupling constants can be bounded locally uniformly in the coefficients of the quadratic forms $\bq$.
Although it is possible to obtain such uniform bounds by keeping track of the dependence on $\bq$ in all our proofs, we use this opportunity to record a compactness argument for decoupling constants whose validity is not restricted to quadratic forms.

\begin{theorem}
\label{thm:main-uniform}
For every $2\leq q \leq p < \infty$, $\epsilon>0$, and real quadratic forms $Q_{1},\dotsc,Q_{n}$ in $d$ variables, there exist $C_{\bq,\epsilon,q,p}<\infty$ and a neighborhood $\calQ$ of $(Q_{1},\dotsc,Q_{n})$ such that, for every $(\tilde{Q}_{1},\dotsc,\tilde{Q}_{n}) \in \calQ$ and every $\delta \in (0,1)$, we have
\[
\Dec_{q, p}((\tilde{Q}_{1},\dotsc,\tilde{Q}_{n}),\delta)
\leq C_{\bq,\epsilon,q,p} \delta^{-\Gamma_{q,p}(\bq)-\epsilon},
\]
where $\Gamma_{q,p}(\bq)$ is given by \eqref{bestconstant}.
\end{theorem}

\begin{lemma}[Affine rescaling]
\label{lem:affine-scaling}
Let $2 \leq q \leq p < \infty$.
For any dyadic numbers $0< \delta \leq \sigma \leq 1$ and every $J \in \Part{\sigma}$,
\begin{equation}
\norm{ \sum_{\Box \in \Part[J]{\delta}}f_{\Box}}_{L^p} \leq \Dec_{q,p}(\bq, \delta/\sigma) ( \sum_{\Box \in \Part[J]{\delta}} \norm{ f_{\Box}}_{L^p}^q )^{1/q}.
\end{equation}
\end{lemma}
Such a lemma has also been standard in the decoupling literature, see for instance \cite[Section 4]{MR3374964} or \cite[Lemma 1.23]{MR4143735}.

\begin{proof}[Proof of Theorem~\ref{thm:main-uniform}]
Let $\sigma = \sigma(\bq,\epsilon,q,p)$ be a small number, which will be determined later.
We may assume that $\delta \leq \sigma/4$.
We consider a tuple $(\tilde{Q}_{1},\dotsc,\tilde{Q}_{n})$ such that
\begin{equation}\label{201020e5_5}
\sup_i \norm{\Hess(\tilde{Q}_i-{Q}_i)} < \sigma^{2}/(10d+10).
\end{equation}
Then, for every $J \in \calP(\sigma)$ and $\Box \in \Part[J]{\delta}$, we have
\begin{equation}
\label{eq:U-tildeQ}
\calU_{\Box}(\tilde{\bq}) \subseteq \calU_{J}(\bq).
\end{equation}

Take a collection of functions $f_{\Box}$ with $\supp \hat{f}_{\Box}\subset \calU_{\Box}(\tilde{\bq})$ for each $\Box\in \calP(\delta)$.
Using \eqref{eq:U-tildeQ} and the definition of $\Gamma_{q,p}(\bq)$, we obtain
\begin{equation}
\begin{split}
\norm{ \sum_{\Box \in \Part{\delta}}f_{\Box}}_{L^p}
&\leq {C}_{\bq,\epsilon}\sigma^{-\Gamma_{q,p}(\bq)-\epsilon/2} ( \sum_{J \in \Part{\sigma}} \norm{ f_{J}}_{L^p}^q )^{1/q}
\\& \leq {C}_{\bq,\epsilon}
\sigma^{-\Gamma_{q,p}(\bq)-\epsilon/2}
\Dec_{q,p}(\tilde{\bq}, \delta/\sigma)
( \sum_{\Box \in \Part{\delta}} \norm{ f_{\Box}}_{L^p}^q )^{1/q}.
\end{split}
\end{equation}
The last inequality follows from the affine rescaling (Lemma~\ref{lem:affine-scaling}).
Hence, we obtain
\begin{equation}
\Dec_{q,p}(\tilde{\bq},\delta) \leq
{C}_{\bq,\epsilon}
\sigma^{-\Gamma_{q,p}(\bq)-\epsilon/2}
\Dec_{q,p}(\tilde{\bq}, \delta/\sigma).
\end{equation}
We iterate this inequality $\log_{\sigma^{-1}}(\delta^{-1})$-times and obtain
\begin{equation}
\Dec_{q,p}(\tilde{\bq},\delta) \leq
C_{\bq, \epsilon}
(\delta^{-1})^{\log_{\sigma^{-1}}{{C}_{\bq,\epsilon}}}\delta^{-\Gamma_{q,p}(\bq)-\epsilon/2}.
\end{equation}
It suffices to take $\sigma$ small enough so that $\log_{\sigma^{-1}} {C}_{\bq,\epsilon} \leq \epsilon/2$.
\end{proof}

\begin{corollary}\label{201029coro5_4}
Let $2\le q\le p<\infty$.
For each $\epsilon>0$, there exists $C_{\bq, \epsilon, q, p}<\infty$ such that, for every linear subspace $H \subset \R^{d}$ of co-dimension one, we have
\begin{equation}
\calD_{q, p}(\bq|_H, \delta)\le C_{\bq, \epsilon, q, p} \delta^{-\Lambda-\epsilon},
\end{equation}
where $\Lambda$ was defined in \eqref{eq:lower-dim-dec-exponent}.
\end{corollary}
\begin{proof}
Recall from \eqref{eq:Q-on-subspace} that the $\bq|_{H}$ are parametrized by the orthogonal group $O(d)$.
Since $\bq|_{H}$ depends continuously on the rotation used to define it, the group $O(d)$ is compact, and by Theorem~\ref{thm:main-uniform}, we obtain the claim.
\end{proof}

To prepare for the broad-narrow analysis of Bourgain and Guth \cite{MR2860188} in the following section, we need the following lemma that takes care of the case when frequency cubes are clustered near sub-varieties of low degrees.

\begin{lemma}[{\cite[Corollary 2.18]{MR4143735}}]\label{201030lemma5_5}
For every $d\ge 1$, $D>1$ and $\epsilon>0$, there exists $c=c(D, \epsilon)>0$ such that the following holds.
For every sufficiently large $K$, there exist
\begin{equation}\label{201030e5_10}
K^c\le K_1\le K_2\le \dots \le K_D\le \sqrt{K}
\end{equation}
such that for every non-zero polynomial $P$ in $d$ variables of degree at most $D$, there exist collections of pairwise disjoint cubes $\calW_j\subset \calP(1/K_j)$, $j=1, 2, \dots, D$, such that
\begin{equation}
\calN_{1/K}(Z_P)\cap [0, 1]^d\subset \bigcup_{j=1}^D \bigcup_{W\in \calW_j} W
\end{equation}
and
\begin{equation}
\norm[\Big]{\sum_{W\in \calW_j} f_W}\lesssim_{D, \bq, \epsilon, q, p} K_j^{\Lambda+\epsilon} \Big(\sum_{W\in \calW_j} \norm{f_W}_p^q \Big)^{1/q}.
\end{equation}
Here $\calN_{1/K}(Z_P)$ denotes the $1/K$ neighborhood of the zero set of $P$.
\end{lemma}
This lemma was stated in \cite{MR4143735} only for $p=q$ and with Fourier support condition that is slightly different from \eqref{eq:Uncertainty-box}.
The same proof works also for $q\le p$ and with Fourier support condition \eqref{eq:Uncertainty-box} without any change, and we will therefore not repeat it here.
The main hypothesis of \cite[Corollary 2.18]{MR4143735} is \cite[Hypothesis 2.4]{MR4143735}, which is exactly what we verified in Corollary~\ref{201029coro5_4}.

\subsection{Multilinear decoupling}\label{subsection_multi}
For a positive integer $K$, a transversality parameter $\theta>0$, and $0 < \delta < K^{-1}$, the \emph{multilinear decoupling constant}
\begin{equation}
\MulDec(\delta, \theta, K)=\MulDec(\bq, \delta, \theta, K)
\end{equation}
is the smallest constant such that the inequality
\begin{multline}
\label{eq:multilin-Dec}
\Bigl( \int_{\R^{d+n}} \bigl( \avprod_{j=1}^{M} \norm{ f_{W_j} }_{\avL^{p}(B(x,K))} \bigr)^{p} \dif x \Bigr)^{1/p}\\
\le \MulDec_{q,p}(\bq, \delta, \theta, K)
\avprod_{j=1}^{M} \Bigl( \sum_{\Box \in \Part[W_j]{\delta}} \norm{f_{\Box}}_{L^p(\R^{d+n})}^{q} \Bigr)^{\frac{1}{q}}
\end{multline}
holds for every choice of functions $f_{\Box}$ and every $\theta$-uniform set $\Set{W_{1},\dotsc,W_{M}} \subseteq \Part{K^{-1}}$ with $1\leq M \leq K^{d}$.

We use a version of the Bourgain--Guth reduction of linear to multilinear estimates \cite{MR2860188}.
Estimates of a similar form already appeared in works of Bourgain and Demeter, see \cite{MR3736493} and \cite{MR3614930} for instance.
The version below is a minor variant of \cite[Proposition 2.33]{MR4143735}.
This is the place where the uniform bound in Theorem~\ref{thm:main-uniform} is used.
\begin{proposition}
\label{prop:Bourgain-Guth}
Let $2 \leq q \leq p < \infty$.
Let $\Lambda$ be given by \eqref{eq:lower-dim-dec-exponent}.
Then, for each $\epsilon>0$ and $\theta>0$, there exists $K$ such that
\begin{equation}
\Dec_{q,p}(\bq, \delta)
\lesssim_{\epsilon, \theta}
\delta^{-\Lambda-\epsilon}
+ \delta^{-\epsilon} \max_{\delta\le \delta'\le 1; \delta' \text{dyadic}} \Big[\big(\frac{\delta'}{\delta} \big)^{\Lambda} \MulDec_{q,p}(\bq, \delta', \theta, K)\Big].
\end{equation}
\end{proposition}
\begin{proof}[Proof of Proposition~\ref{prop:Bourgain-Guth}]
Let $\Set{f_{\Box}}_{\Box\in \calP(\delta)}$ be a collection of functions with $\supp\hat{f}_{\Box}\subset \calU_{\Box}$.
In the proof, for each dyadic cube $J$ with $l(J)\ge \delta$, we denote
\begin{equation}
f_{J}:=\sum_{\Box\in \calP(J, \delta)} f_{\Box}.
\end{equation}
Let $K$ be a large constant that is to be determined.
For each ball $B'\subset \R^{d+n}$ of radius $K$, we initialize
\begin{equation}\label{211127e5.17}
\calS_0(B'):=\Set{W\in \calP(1/K)\big|\ \ \norm{f_W}_{L^p(B')}\ge K^{-d} \max_{W'\in \calP(1/K)} \norm{f_{W'}}_{L^p(B')}}.
\end{equation}
We repeat the following algorithm.
Let $\iota\ge 0$.
If $\calS_{\iota}(B')=\emptyset$ or $\calS_{\iota}(B')$ is $\theta$-uniform, then we set
\begin{equation}
\calT(B'):=\calS_{\iota}(B')
\end{equation}
and terminate.
Otherwise, there exists a sub-variety $Z$ of degree at most $d$ such that
\begin{equation}
\label{eq:large-portion-of-Si}
|\Set{W\in \calS_{\iota}(B')| 2W\cap Z\neq \emptyset }|\ge \theta |\calS_{\iota}(B')|.
\end{equation}
Fix any such variety $Z$.
Note that $2W\cap Z \neq \emptyset \implies W \subseteq \calN_{2^{d}/K}(Z)$.
For $j \in \Set{1, \dotsc, d}$, let $\calW_{\iota, j}(B'):=\calW_j$ be as in Lemma~\ref{201030lemma5_5} with $K$ replaced by $K/2^{d}$.
Repeat this algorithm with
\begin{equation}
\calS_{\iota+1}(B'):=\calS_{\iota}(B')\setminus \bigcup_{j=1}^d \bigcup_{W\in \calW_{\iota, j}(B')} \calP(W, 1/K).
\end{equation}
This algorithm terminates after $O(\log K)$ steps, with an implicit constant depending on $\theta$, as in each step we remove at least the set on the left-hand side of \eqref{eq:large-portion-of-Si}, which constitutes a fixed proportion $\theta$ of $\calS_{\iota}(B')$.

To process the cubes in $\calW_{\iota, j}$ and to avoid multiple counting, we define
\begin{equation}
\widetilde{\calW}_{\iota, j}:=\Big( \calW_{\iota, j}\setminus \bigcup_{0\le \iota'<\iota} \calW_{\iota', j}\Big)\setminus \bigcup_{1\le j'< j}\bigcup_{\iota'} \bigcup_{W\in \calW_{\iota', j'}} \calP(W, 1/K_j).
\end{equation}
So far we see that every cube in \eqref{211127e5.17} can be covered by exactly one cube in
\begin{equation}
\Big(\bigcup_{\iota}\bigcup_{j} \widetilde{\calW}_{\iota, j}\Big)\bigcup \mc{T}(B').
\end{equation}
Therefore, by the triangle inequality we obtain
\begin{equation}\label{201031e5_22}
\begin{split}
& \norm{\sum_{\Box\subset [0, 1]^d} f_{\Box}}_{L^p(B')}\le \Big( \sum_{W\in \calP(1/K)} \norm{f_{W}}_{L^p(B')}^q\Big)^{1/q}\\
& + \sum_{\iota\lesssim \log K} \sum_{j=1}^d \norm[\Big]{\sum_{W\in \widetilde{\calW}_{\iota, j}} f_{W} }_{L^p(B')}+ \sum_{W\in \calT(B')}\norm{f_W}_{L^p(B')}.
\end{split}
\end{equation}
On the right hand side of \eqref{201031e5_22}, the first term is used to take care of the cubes that are not counted in \eqref{211127e5.17}.
Next, we will see how to handle all these three terms.
The second term on the right hand side will be processed via a standard localization argument (see for instance \cite[Remark 1.24]{MR4143735}) and Lemma~\ref{201030lemma5_5}.
It is bounded by
\begin{equation}\label{201031e5_23}
C_{\bq, \epsilon, p, q} \log K \sum_{j=1}^d K_j^{\Lambda+\epsilon} \Big(\sum_{W\in \calP(1/K_j)} \norm{f_W}_{L^p(w_{B'})}^q \Big)^{1/q}.
\end{equation}
Recall in Lemma~\ref{201030lemma5_5} that $K^c\le K_j\le \sqrt{K}$ for some $c=c(d, \epsilon)$ and every $j$.
This allows us to absorb $\log K$ by $K_j^{\epsilon}$, which is the only place where the lower bound $K^c$ in \eqref{201030e5_10} is used.
To bound the last term, we use the definition of $\calT(B')$ and obtain
\begin{equation}\label{201031e5_24}
K^{d} \max_{W\in \calT(B')} \norm{f_{W}}_{L^p(B')} \le K^{2d} \Big( \sum_{\substack{\Set{W_1, \dots, W_M}\subseteq \calP(1/K)\\ \theta-\text{uniform}}} \avprod_{j=1}^M\norm{f_{W_j}}^p_{L^p(B')}\Big)^{1/p}.
\end{equation}
The above estimate seems rather crude, but we can allow any $K$-dependent constant in the estimate for this term.
We plug \eqref{201031e5_23} and \eqref{201031e5_24} in \eqref{201031e5_22}, integrate over the centers of balls $B'$, and obtain
\begin{equation}\label{201031e5_25}
\begin{split}
& \norm{\sum_{\Box\in \calP(\delta)} f_{\Box}}_{L^p(\R^{d+n})}\le C_{\bq, \epsilon, q, p} \sum_{j=0}^d K_j^{\Lambda+2\epsilon}\Big( \sum_{W\in \calP(1/K_j)} \norm{f_{W}}_{L^p(\R^{d+n})}^q\Big)^{1/q}\\
& + K^{2d} \sum_{\substack{\Set{W_1, \dots, W_M} \subseteq \calP(1/K)\\ \theta-\text{uniform}}} \Big( \sum_{B'\subset \R^{d+n}} \avprod_{j=1}^M\norm{f_{W_j}}^p_{L^p(B')}\Big)^{1/p}.
\end{split}
\end{equation}
Here we let $K_0:=K$.
The terms under the sum in the former term have the same form as that on the left hand side, and therefore are ready for an iteration argument. In other words, we will apply (rescaled versions of ) \eqref{201031e5_25} to each term $\norm{f_W}_{L^p(\R^{d+n})}$ under the sum in the former term.
By the definition of the multi-linear decoupling constant, the latter term can be controlled by
\begin{equation}\label{201031e5_26}
K^{2d} 2^{K^{d}} \MulDec_{q, p}(\bq, \delta, \theta, K) \Big( \sum_{\Box\in \calP(\delta)} \norm{f_{\Box}}_{L^p(\R^{d+n})}^q\Big)^{1/q},
\end{equation}
where we used that there are only $2^{K^{d}}$ subsets of $\calP(1/K)$, and hence at most that many $\theta$-uniform subsets.
We plug \eqref{201031e5_26} in \eqref{201031e5_25}.
Now it is standard argument to iterate \eqref{201031e5_25} and obtain the desired estimate in the proposition.
We leave out the details and refer to \cite[Section 8]{MR3614930} or \cite[Proposition 8.4]{MR3614930}.
\end{proof}

Recall that we have assumed \eqref{eq:Gamma>Lambda}.
For most of Section~\ref{sec:induction-on-scales}, we fix some $0<\epsilon<\Gamma-\Lambda$, a transversality parameter $\theta>0$, and a corresponding $K$ as in Proposition~\ref{prop:Bourgain-Guth}.

The mutlilinear decoupling constant will be estimated by the same procedure as in \cite{MR3374964}.
For a detailed exposition of this argument we refer to \cite[Theorem 10.16]{MR3592159} or \cite[Section 2.6]{MR4143735}.
We use a compressed version of this argument, in which each step is expressed as an inequality between the quantities \eqref{eq:tilde-a} below.
This version of the Bourgain--Demeter argument was originally motivated by decoupling for higher degree polynomials, see \cite{MR4031117}.

For a $\theta$-uniform set $\Set{ W_{j} }_{j=1}^{M} \subset \Part{1/K}$ and a choice of functions $f_{\Box}$, $\Box \in \Part{\delta}$, we write
\begin{align*}
\tA_{2}(b)
&:=
\norm[\Big]{\avprod_{j=1}^{M} \Big( \sum_{J \in \Part[W_{j}]{\delta^{b}}} \norm{ f_{J} }^2_{\avL^{2}(w_{B(x,\delta^{-2b})})}\Big)^{1/2}}_{L^{p}_{x \in \R^{d+n}}},
\\
\tA_{t}(b)
&:=
\norm[\Big]{ \avprod_{j=1}^{M} \Big(\sum_{J \in \Part[W_{j}]{\delta^{b}}}\norm{ f_{J} }^{\tilde{q}}_{\avL^{t}(w_{B(x,\delta^{-b})})}\Big)^{1/\tilde{q}}}_{L^{p}_{x \in \R^{d+n}}},
\\
\tA_{p}(b)
&:=
\norm[\Big]{ \avprod_{j=1}^{M} \Big(\sum_{J \in \Part[W_{j}]{\delta^{b}}} \norm{ f_{J} }^q_{\avL^{p}(w_{B(x,\delta^{-2b})})}\Big)^{1/q}}_{L^{p}_{x \in \R^{d+n}}},
\end{align*}
where $0<b\leq 1$ and
\begin{equation}
\label{eq:exponent-t}
\frac{1}{t} = \frac{1/2}{p} + \frac{1/2}{2},
\quad
\frac{1}{\tilde{q}} = \frac{1/2}{q} + \frac{1/2}{2}.
\end{equation}
Note that $2 \leq \tilde{q} \leq t \leq p$.
For
\begin{equation}
0<b<1 \text{ and } *=2,t,p,
\end{equation}
let $a_{*}(b)$ be the infimum over all exponents $a$ such that, for every $\theta$-uniform set $\Set{W_{j}}_{j=1}^{M} \subseteq \Part{1/K}$, every $\delta < 1/K$, and every choice of functions $f_{\Box}$, $\Box\in\Part{\delta}$, we have
\begin{equation}
\label{eq:a_*}
\tA_{*}(b) \lesssim_{a,\theta,K} \delta^{-a}
\avprod_{j=1}^{M} \Big( \sum_{\Box \in \Part[W_j]{\delta}} \norm{f_{\Box}}^q_{L^p(\R^{d+n})}\Big)^{1/q}
\end{equation}
with the implicit constant independent of the choice of the tuples $(W_{j})$ and $(f_{\Box})$, and in particular independent of $b$ as we will send $b\to 0$.
It follows from Hölder's inequality that this $a_{*}(b) < \infty$.
Recall that $\Gamma := \Gamma_{q,p}(\bq)$.
As in \cite[Section 3.6]{MR4031117}, we define
\begin{equation}
\label{eq:tilde-a}
a_{*} := \liminf_{b \to 0} \frac{\Gamma - a_{*}(b)}{b},
\quad
* \in \Set{2,t,p}.
\end{equation}
The next lemma says that $a_{*}$ is non-trivial.
\begin{lemma}\label{202201lem5_7}
Under the above notation, it holds that
\begin{equation}
\label{eq:1}
a_{*} < \infty,
\end{equation}
for $*=2,t,p$.
\end{lemma}
\begin{proof}[Proof of Lemma~\ref{202201lem5_7}]
By H\"older's inequality and Bernstein's inequality, the left-hand side of \eqref{eq:multilin-Dec} is bounded by
\begin{equation}\label{201101e5_31}
\delta^{-Cb} \tA_{*}(b)
\end{equation}
for any $*\in\Set{2,t,p}$ and any $0<b<1$ with some constant $C$ depending on $*$.
Therefore, we obtain that
\begin{equation}
\MulDec_{q,p}(\bq, \delta, \theta, K) \lesssim_{\epsilon, \theta, K} \delta^{-\big(Cb+a_{*}(b)+\epsilon\big)},
\end{equation}
for every $\epsilon>0$ and $1>b>0$.
This, together with Proposition~\ref{prop:Bourgain-Guth} and the assumption \eqref{eq:Gamma>Lambda}, implies that
\begin{equation}
\Gamma\le Cb+a_{*}(b).
\end{equation}
This finishes the proof of the lemma.
\end{proof}

\subsection{Using linear decoupling}
By H\"older's inequality, we obtain
\begin{equation}
\tA_{p}(b) \leq
\avprod_{j=1}^{M} \norm[\Big]{ \Big(\sum_{J \in \Part[W_{j}]{\delta^{b}}} \norm{ f_{J} }^q_{\avL^{p}(w_{B(x,\delta^{-2b})})}\Big)^{1/q}}_{L^{p}_{x \in \R^{d+n}}}.
\end{equation}
By Minkowski's inequality, this is further bounded by
\begin{equation}
\begin{split}
& \avprod_{j=1}^{M} \Big(\sum_{J \in \Part[W_{j}]{\delta^{b}}} \norm[\Big]{ \norm{ f_{J} }_{\avL^{p}(w_{B(x,\delta^{-2b})})}}_{L^{p}_{x \in \R^{d+n}}}^q \Big)^{1/q}\\
& \lesssim \avprod_{j=1}^{M} \Big(\sum_{J \in \Part[W_{j}]{\delta^{b}}} \norm{ f_{J} }_{L^{p}(\R^{d+n})}^q \Big)^{1/q}.
\end{split}
\end{equation}
By the definition of the decoupling exponent and affine scaling (Lemma~\ref{lem:affine-scaling}), this is
\begin{equation}
\lesssim_{\epsilon} \delta^{-(\Gamma+\epsilon)(1-b)} \avprod_{j=1}^M \Big( \sum_{\Box\in \calP(W_j, \delta)} \norm{f_{\Box}}_{L^p}^q\Big)^{1/q}.
\end{equation}
Hence
\[
a_{p}(b) \leq (\Gamma+\epsilon)(1-b),
\]
for every $\epsilon>0$, which means $a_p(b)\le \Gamma(1-b)$.
It follows that
\begin{equation}
\label{eq:2}
a_{p} \geq \Gamma.
\end{equation}

\subsection{Using \texorpdfstring{$L^{2}$}{L2} orthogonality}
By $L^{2}$ orthogonality, see e.g.\ \cite[Appenidx B]{MR4031117} for details, we have
\begin{equation}
\begin{split}
\tA_2(b)& =\norm[\Big]{\avprod_{j=1}^{M} \Big( \sum_{J \in \Part[W_{j}]{\delta^{b}}} \norm{ f_{J} }^2_{\avL^{2}(w_{B(x,\delta^{-2b})})}\Big)^{1/2}}_{L^{p}_{x \in \R^{d+n}}}\\
& \lesssim \norm[\Big]{\avprod_{j=1}^{M} \Big( \sum_{J \in \Part[W_{j}]{\delta^{2b}}} \norm{ f_{J} }^2_{\avL^{2}(w_{B(x,\delta^{-2b})})}\Big)^{1/2}}_{L^{p}_{x \in \R^{d+n}}}.
\end{split}
\end{equation}
We further apply H\"older's inequality and obtain
\begin{equation}
\lesssim \delta^{-d\cdot 2b(1/2-1/\tilde{q})} \norm[\Big]{\avprod_{j=1}^{M} \Big( \sum_{J \in \Part[W_{j}]{\delta^{2b}}} \norm{ f_{J} }^{\tilde{q}}_{\avL^{t}(w_{B(x,\delta^{-2b})})}\Big)^{1/\tilde{q}}}_{L^{p}_{x \in \R^{d+n}}}
\end{equation}
Note that the last expression is exactly $\delta^{-d b(1-2/\tilde{q})} \tA_{t}(2b)$.
Hence
\[
a_{2}(b) \leq d b(1-2/\tilde{q}) + a_{t}(2b).
\]
It follows that
\begin{equation}
\label{eq:4}
a_{2} \geq - d(1-2/\tilde{q}) + 2 a_{t}.
\end{equation}

\subsection{Ball inflation}
Using Corollary~\ref{cor:ball-inflation:lqLt} with $\rho=\delta^b$ and taking $L^p$ norms in $x_0$ on both sides of \eqref{eq:ball-inflation:lqLt}, we obtain
\begin{equation}\label{201031e5_39}
\begin{split}
\tA_{t}(b)& =\norm[\Big]{ \avprod_{j=1}^{M} \Big(\sum_{J \in \Part[W_{j}]{\delta^{b}}}\norm{ f_{J} }^{\tilde{q}}_{\avL^{t}(w_{B(x,\delta^{-b})})}\Big)^{1/\tilde{q}}}_{L^{p}_{x \in \R^{d+n}}}\\
& \lesssim_{\epsilon}
\delta^{-b(\gamma+\epsilon)} \norm[\Big]{ \avprod_{j=1}^{M} \Big(\sum_{J \in \Part[W_{j}]{\delta^{b}}}\norm{ f_{J} }^{\tilde{q}}_{\avL^{t}(w_{B(x,\delta^{-2b})})}\Big)^{1/\tilde{q}}}_{L^{p}_{x \in \R^{d+n}}},
\end{split}
\end{equation}
for every $\epsilon>0$, where
\begin{equation}
\begin{split}
\gamma&:=\frac{d}{t} - \frac{d+n}{p} + \frac{\kappa((1-\theta)p/t)}{p}
\\ &\le
\frac{d}{t} - \frac{d+n}{p} + \frac{1}{p} \sup_{0\leq n'\leq n} \bigl( n' + (1 - \frac{p}{t} (1-\theta)) \numvar_{n'}(\bq) \bigr),
\end{split}
\end{equation}
and the log factors in Corollary~\ref{cor:ball-inflation:lqLt} have been absorbed by $\delta^{-b\epsilon}$.
In the last step we used Corollary~\ref{cor:kappa-bound}.
In the end, we apply H\"older's inequality to the last term in \eqref{201031e5_39} and obtain
\begin{equation}
\tA_t(b)\lesssim \delta^{-b(\gamma+\epsilon)} \tA_{2}(b)^{1/2} \tA_{p}(b)^{1/2}.
\end{equation}
It follows that
\[
a_{t}(b) \leq
b\gamma + a_{p}(b)/2 + a_{2}(b)/2.
\]
Substituting this inequality into the definition \eqref{eq:tilde-a}, we obtain
\begin{equation}
\label{eq:5}
a_{t} \geq
-\gamma + a_{p}/2 + a_{2}/2.
\end{equation}

\subsection{Proof of Theorem~\ref{thm:main-inductive-step}}
Inequalities \eqref{eq:exponent-t}, \eqref{eq:2}, \eqref{eq:4}, \eqref{eq:5} imply
\[
\Gamma
\leq
a_{p}
\leq
2\gamma - a_{2} + 2a_{t}
\leq
2\gamma + d(1-2/\tilde{q})
=
2\gamma + d(1/2-1/q).
\]
Inserting the definitions of the respective terms into this inequality, we obtain
\[
\Gamma_{q,p}(\bq)
\leq
2 \Bigl( \frac{d}{t} - \frac{d+n}{p} + \frac{1}{p} \sup_{0\leq n'\leq n} \bigl( n' + (1 - \frac{p}{t} (1-\theta)) \numvar_{n'}(\bq) \bigr) \Bigr)
+ d \bigl( \frac{1}{2} - \frac{1}{q} \bigr).
\]
Both sides of this inequality depend continuously on $\theta$, and we consider its limit when $\theta\to 0$.
This gives
\[
\Gamma_{q,p}(\bq)
\leq
2 \Bigl( \frac{d}{t} - \frac{d+n}{p} + \sup_{0\leq n'\leq n} \bigl( \frac{n'}{p} + (\frac{1}{p} - \frac{1}{t} ) \numvar_{n'}(\bq) \bigr) \Bigr)
+ d \bigl( \frac{1}{2} - \frac{1}{q} \bigr).
\]
Substituting the ansatz \eqref{eq:exponent-t} for $t$, we obtain
\begin{align*}
\Gamma_{q,p}(\bq)
&\leq
d(\frac{1}{p}+\frac{1}{2}) - 2\frac{d+n}{p} + \sup_{0\leq n'\leq n} \bigl( \frac{2n'}{p} + (\frac{2}{p} - \frac{1}{p} - \frac{1}{2} ) \numvar_{n'}(\bq) \bigr)
+ d \bigl( \frac{1}{2} - \frac{1}{q} \bigr)
\\ &=
\sup_{0 \leq n'\leq n} \bigl( \bigl( \frac{1}{2}-\frac{1}{p} \bigr) (d-\numvar_{n'}) - \frac{2(n - n')}{p} \bigr) + d \bigl( \frac{1}{2}-\frac{1}{q} \bigr).
\end{align*}
This finishes the proof of Theorem~\ref{thm:main-inductive-step}.

\section{Lower bounds in Theorem~\ref{thm:main}}
\label{sec:lower-bd}

In this section, we show the lower bounds for $\ell^q L^p$ decoupling constants in Theorem~\ref{thm:main} for $q\le p$.
We will prove that
\begin{equation}
\label{eq:dec-exponent}
\Gamma_{q,p}(\bq) \ge
\max \Bigl( \sup_{H}\Gamma_{q,p}(\bq|_H) , \max_{0\leq n'\leq n} \Big( d\bigl(1-\frac{1}{p}-\frac{1}{q} \bigr) - \numvar_{n'}(\bq)\bigl(\frac{1}{2}-\frac{1}{p}\bigr) - \frac{2(n-n')}{p}\Big) \Bigr),
\end{equation}
where $H$ is a hyperplane passing through the origin, for every $p\ge 2, q\ge 2$.
Note that here we do not necessarily require $q\le p$.
One can apply the above inequality repeatedly and then obtain the lower bounds in Theorem~\ref{thm:main} for $q \leq p$.

First of all, we relate the decoupling exponent with the decoupling exponents on subspaces.
Here, the distinction between the cases $q\leq p$ and $q>p$ becomes apparent.
\begin{lemma}
\label{lem:subspace-lower-bd}
Let $\bq$ be an $n$-tuple of quadratic forms in $d$ variables and $H \subseteq \R^{d}$ a linear subspace of dimension $d'$.
Then, for any $2 \leq q \leq p < \infty$, we have
\begin{equation}
\label{eq:subspace-lower-bd:q<p}
\Gamma_{q,p}(\bq) \geq \Gamma_{q,p}(\bq|_{H}),
\end{equation}
and, for any $2 \leq p < q \leq \infty$, we have
\begin{equation}
\label{eq:subspace-lower-bd:q>p}
\Gamma_{q,p}(\bq) \geq \Gamma_{q,p}(\bq|_{H}) + (d-d') \bigl( \frac{1}{p} - \frac{1}{q} \bigr).
\end{equation}
\end{lemma}
\begin{proof}[Proof of Lemma~\ref{lem:subspace-lower-bd}]
For notational convenience, assume that $\R^{d} = H \times \R^{d''}$ with $d''=d-d'$.
The bound \eqref{eq:subspace-lower-bd:q<p} will follow from
\begin{equation}
\label{eq:subspace-lower-bd-delta:q<p}
\Dec_{q,p}(\bq,C\delta) \gtrsim \Dec_{q,p}(\bq|_{H},\delta),
\end{equation}
for some absolute constant $C$.
To see this, let $\Set{ \tilde{f}_{\Box'} \given \Box' \in \Part[{[0,1]^{d'}}]{\delta} }$, be a tuple of functions on $\R^{d' + n}$ that nearly extremizes the inequality \eqref{equ:decoupling_definition} for $\bq|_{H}$.
Fix a bump function $\phi$ such that $\supp \widehat{\phi} \subseteq B(0,\delta^{2}) \subset \R^{d''}$ and, for
\[
\Box = \Box' \times \Box'' \in \Part[{[0,1]^{d}}]{\delta} = \Part[{[0,1]^{d'}}]{\delta} \times \Part[{[0,1]^{d''}}]{\delta},
\]
consider the functions
\[
f_{\Box} = f_{\Box' \times \Box''}
=
\begin{cases}
\tilde{f}_{\Box'} \otimes \phi, & \Box'' = \Box''_{0} := [0,\delta]^{d''},\\
0, & \Box'' \neq \Box''_{0}.
\end{cases}
\]
Then $\supp \widehat{f_{\Box}} \subseteq C \calU_{\Box}$ and
\[
\norm{\sum_{\Box} f_{\Box}}_{p} = \norm{\phi}_{p} \norm{\sum_{\Box'} f_{\Box'}}_{p},
\quad
\norm{f_{\Box' \times \Box''}}_{p} = \one_{\Box'' = \Box''_{0}} \norm{\phi}_{p} \norm{f_{\Box'}}_{p},
\]
which implies \eqref{eq:subspace-lower-bd-delta:q<p}.
Here $\one$ denotes an indicator function, that takes the value $1$ if the statement in the subscript is true, and $0$ otherwise.

To see \eqref{eq:subspace-lower-bd:q>p}, we define $f_{\Box' \times \Box''_{0}}$, as above.
For other $\Box'' \in \Part[{[0,1]^{d''}}]{\delta}$, let $a''\in\R^{d''}$ be the center of $\Box''$ and define
\[
f_{\Box' \times \Box''} := A_{\Box''} f_{\Box' \times \Box''_{0}}(\cdot + c_{\Box''}),
\]
where $c_{\Box''} \in \R^{d+n}$ are very large vectors and the linear operators $A_{\Box''}$ are given by affine transformations in the Fourier space:
\[
\widehat{A_{\Box''} f}(\xi,\eta) := \widehat{f}(\xi-(0,a''),\eta+Q(0,a'')-\nabla \bq(0,a'') \cdot \xi).
\]
If $c_{\Box''}$ are sufficiently far apart, then functions $f_{\Box' \times \Box''}$ and $f_{\tilde\Box' \times \tilde\Box''}$ are almost disjointly supported for $\Box'' \neq \tilde\Box''$, so that
\begin{align*}
\norm{\sum_{\Box} f_{\Box}}_{p}
&\sim
\bigl( \sum_{\Box''} \norm{\sum_{\Box'} f_{\Box' \times \Box''}}_{p}^{p} \bigr)^{1/p}
\\ &=
\norm{\phi}_{p} \bigl( \sum_{\Box''} \norm{\sum_{\Box'} \tilde{f}_{\Box'}}_{p}^{p} \bigr)^{1/p}=
\norm{\phi}_{p} \delta^{-d''/p} \norm{\sum_{\Box'} \tilde{f}_{\Box'}}_{p}
\end{align*}
and
\begin{align*}
\bigl( \sum_{\Box} \norm{f_{\Box}}_{p}^{q} \bigr)^{1/q}
&=
\bigl( \sum_{\Box''} \sum_{\Box'} \norm{f_{\Box' \times \Box''}}_{p}^{q} \bigr)^{1/q}
\\ &=
\norm{\phi}_{p} \bigl( \sum_{\Box''} \sum_{\Box'} \norm{\tilde{f}_{\Box'}}_{p}^{q} \bigr)^{1/q}
=
\norm{\phi}_{p} \delta^{-d''/q} \bigl(\sum_{\Box'} \norm{\tilde{f}_{\Box'}}_{p}^{q} \bigr)^{1/q}.
\end{align*}
This implies
\[
\Dec_{q,p}(\bq,C\delta) \gtrsim \delta^{-d''(1/p-1/q)} \Dec_{q,p}(\bq|_{H},\delta),
\]
and therefore \eqref{eq:subspace-lower-bd:q>p}.
\end{proof}

To show the lower bound in \eqref{eq:dec-exponent}, it remains to prove

\begin{proposition}
\label{prop:skew-lower-bd}
Let $\bq$ be an $n$-tuple of quadratic forms in $d$ variables.
For $0\leq n'\leq n$ and $2 \leq q, p \leq \infty$, we have
\begin{equation}
\label{eq:skew-lower-bd}
\Gamma_{q,p}(\bq) \geq
d \bigl(1-\frac{1}{p}-\frac{1}{q} \bigr) - \numvar_{n'}(\bq)\bigl(\frac{1}{2}-\frac{1}{p}\bigr) - \frac{2(n-n')}{p}.
\end{equation}
\end{proposition}

\begin{proof}[Proof of Proposition~\ref{prop:skew-lower-bd}]
Let $d' = \numvar_{n'}(\bq)$.
After linear changes of variables, we may assume that $Q_{1},\dotsc,Q_{n'}$ depend only on $\xi_{1},\dotsc, \xi_{d'}$.
Write frequency points in $\R^{d+n}$ as
\begin{equation}
(\xi',\xi'',\eta',\eta'') \in \R^{d'+d''+n'+n''},
\end{equation}
with
\begin{equation}
\begin{split}
& \xi'=(\xi_1, \dots, \xi_{d'}), \ \xi''=(\xi_{d'+1}, \dots, \xi_{d}),\\
& \eta'=(\eta_1, \dots, \eta_{n'}), \ \eta''=(\eta_{n'+1}, \dots, \eta_{n}),
\end{split}
\end{equation}
and $d'+d''=d, n'+n''=n$.
Similarly, we write spatial points in $\R^{d+n}$ as
\begin{equation}
(x',x'',y',y'') \in \R^{d'+d''+n'+n''}.
\end{equation}
For a dyadic cube $\Box \in \Part{\delta}$, write $\Box = \Box' \times \Box''$ with $\Box'\subset \R^{d'}$ and $\Box''\subset \R^{d''}$.
Choose functions $f_{\Box}$ of the form
\begin{equation}
f_{\Box}(x', x'', y', y'') = g_{\Box'}(x',y')h_{\Box}(x'',y'')
\end{equation}
with the following properties\footnote{Here and below we use $\Box$ instead of $\Box''$ in $h_{\Box}$ as $Q_{n'+1}, \dots, Q_n$ still depend on $\xi'$.}
\begin{enumerate}
\item $\widehat{g_{\Box'}}$ and $\widehat{h_{\Box}}$ are positive smooth functions satisfying
\begin{equation}
\int \widehat{g_{\Box'}} = \int \widehat{h_{\Box}} = 1,
\end{equation}
\item $\widehat{g_{\Box'}}$ is supported on a ball of radius $\approx \delta^2$ contained in
\begin{equation}
\Set{(\xi', \eta'): \xi'\in \delta\cdot \Box', |\eta_1-Q_1(\xi')|\le \delta^2, \dots, |\eta_{n'}-Q_{n'}(\xi')|\le \delta^2},
\end{equation}
where $\delta\cdot \Box'$ is the box of the same center as $\Box'$ and side length $\delta$ times that of $\Box'$.
\item $\widehat{h_{\Box}}$ is supported on a rectangular box of dimensions comparable to
\begin{equation}
\underbrace{\delta^{1} \times \dotsm \times \delta^{1}}_{\text{$d''$ times}} \times
\underbrace{\delta^{2} \times \dotsm \times \delta^{2}}_{\text{$n''$ times}}
\end{equation}
contained in
\begin{equation}
\bigcup_{\xi'\in \delta\cdot \Box'}\Set{(\xi'', \eta''): \xi''\in \Box'', \abs{\eta_{\tilde{n}'}-Q_{\tilde{n}'}(\xi', \xi'')} \le \delta^2, n'<\tilde{n}'\le n}.
\end{equation}
\end{enumerate}
On one hand, by the uncertainty principle,
\begin{equation}
\norm{f_{\Box}}_{p} \sim \delta^{-(2d'+d''+2n)/p},
\end{equation}
and by definition we have
\begin{equation}
\begin{split}
\norm[\big]{ \sum_{\Box\in\Part{\delta}} f_{\Box} }_{p}
& \leq
\Dec_{q,p}(\bq,\delta)
\Bigl( \sum_{\Box\in\Part{\delta}} \norm{f_{\Box}}^{q} \Bigr)^{1/q}\\
& \sim
\Dec_{q,p}(\bq,\delta)
\delta^{-d/q} \delta^{-(2d'+d''+2n)/p}.
\end{split}
\end{equation}
On the other hand, with $U=\Set{(x'',y'')\in\R^{d''}\times\R^{n''} \given \abs{x''}, \abs{y''} \leq 10^{-d-n}/(\sup_{j}\norm{\Hess Q_{j}} + 1)}$, we have
\begin{equation}
\begin{split}
\norm[\big]{ \sum_{\Box\in\Part{\delta}} f_{\Box} }_{p}
&\gtrsim
\inf_{(x'',y'') \in U}
\norm[\big]{ \sum_{\Box\in\Part{\delta}} f_{\Box} }_{L^{p}(\R^{d'} \times \Set{x''} \times \R^{n'} \times \Set{y''})}
\\ &=
\inf_{(x'',y'') \in U}
\norm[\big]{ \sum_{\Box'} c_{\Box',x'',y''} g_{\Box'} }_{L^{p}(\R^{d'} \times \R^{n'})}
\end{split}
\end{equation}
where
\begin{equation}
\label{eq:cbox'x''y''}
c_{\Box',x'',y''} := \sum_{\Box''} h_{\Box' \times \Box''}(x'',y'')=\sum_{\Box''}h_{\Box}(x'', y'')
\end{equation}
satisfies
\begin{equation}
\abs{ c_{\Box',x'',y''} } \sim \delta^{-d''}
\end{equation}
uniformly in $\Box'$ and $(x'',y'') \in U$.
This is because $h_{\Box}(0,0)=1$ and
\[
\abs{h_{\Box}(x'', y'')-h_{\Box}(0,0)}
\leq
\int \abs{e(x''\cdot \xi''+y''\cdot \eta'')-1} \abs{\widehat{h_\Box}(\xi'',\eta'')} \dif\xi'' \dif\eta''
\leq\frac{1}{2} \int \abs{\widehat{h_\Box}(\xi'',\eta'')} \dif\xi'' \dif\eta''
=
\frac{1}{2},
\]
so that all summands in \eqref{eq:cbox'x''y''} are close to $1$.

Let $\phi_{\delta}(\cdot)= \phi(\delta^{2}\cdot)$, where $\phi$ is a fixed positive Schwartz function on $\R^{d'}\times\R^{n'}$ with $\supp \widehat{\phi} \subset B(0,1/10)$.
Then, by H\"older's inequality,
\begin{equation}
\begin{split}
& \norm[\big]{ \sum_{\Box'} c_{\Box',x'',y''} g_{\Box'} }_{L^{p}(\R^{d'}\times\R^{n'})}\\
&\geq
\norm{\phi_{\delta}}_{1/(1/2-1/p)}^{-1}
\norm[\big]{ \phi_{\delta} \sum_{\Box'} c_{\Box',x'',y''} g_{\Box'} }_{L^{2}(\R^{d'}\times\R^{n'})}
\\ &\sim
\delta^{2\cdot (d'+n')(1/2-1/p)} \norm[\big]{ \sum_{\Box'} c_{\Box',x'',y''} \phi_{\delta}\cdot g_{\Box'} }_{L^{2}(\R^{d'}\times\R^{n'})}.
\end{split}
\end{equation}
Since the Fourier supports of $\phi_{\delta} \cdot g_{\Box'}$ are disjoint for different $(\Box')$'s for sufficiently small $\delta$, we obtain
\begin{equation}
\label{eq:orthogonality_in_Box'}
\begin{split}
\norm[\big]{ \sum_{\Box'} c_{\Box',x'',y''} \phi_{\delta}\cdot g_{\Box'} }_{L^{2}(\R^{d'}\times\R^{n'})}
&=
\Bigl( \sum_{\Box'} \abs{ c_{\Box',x'',y''} }^2 \norm[\big]{ \phi_{\delta}\cdot g_{\Box'} }_{L^{2}(\R^{d'}\times\R^{n'})}^{2} \Bigr)^{1/2} \\
&\sim
\delta^{-d'/2} \cdot \delta^{-d''} \cdot \delta^{-2 \cdot (d'+n')/2},
\end{split}
\end{equation}
uniformly in $(x'', y'') \in U$.
Combining the above estimates, we obtain
\begin{equation}
\begin{split}
& \Dec_{q,p}(\bq,\delta) \delta^{-d/q} \delta^{-(2d'+d''+2n)/p}\\
& \gtrsim
\delta^{2\cdot (d'+n')(1/2-1/p)}
\cdot
\delta^{-d'/2} \cdot \delta^{-d''} \cdot \delta^{-2 \cdot (d'+n')/2}.
\end{split}
\end{equation}
This implies
\begin{align*}
\Gamma_{q,p}(\bq)
\ge
d(1-1/q-1/p) - d'(1/2-1/p) - 2(n-n')/p,
\end{align*}
as desired.
\end{proof}

\section{Sharp \texorpdfstring{$\ell^{q}L^{p}$}{\9041\023{}qLp} decoupling inequalities with \texorpdfstring{$q>p$}{q>p}}\label{201031section8}

\begin{proof}[Proof of Theorem~\ref{thm:main} with $q>p$]
The upper bound $\leq$ follows from the Hölder inequality between $\ell^{p}$ and $\ell^{q}$ sums in the definitions of $\Gamma_{q, p}$ and $\Gamma_{p, p}$.

Let us prove the lower bound.
Recall from Corollary~\ref{1015.1.3} that
\begin{equation}
\label{eq:dec-exp-ell-p-L-p-cc}
\Gamma_{p, p}(\bq) =
\max_{d/2 \leq d'\leq d} \max_{0\leq n'\leq n} \Big( (2d' - \numvar_{d',n'}(\bq)) \bigl(\frac{1}{2}-\frac{1}{p}\bigr) - \frac{2(n-n')}{p}\Big).
\end{equation}
We will show that
\begin{equation}\label{201031e7_1}
\Gamma_{q, p}(\bq)\ge \max_{d'\leq d} \max_{0\leq n'\leq n} \Big( (2d' - \numvar_{d',n'}(\bq)) \bigl(\frac{1}{2}-\frac{1}{p}\bigr) - \frac{2(n-n')}{p}\Big)+ d(1/p-1/q)
\end{equation}
via an induction on $d$.
The base case $d=1$ is easy, as quadratic forms depending on one variable $\xi_1$ are all multiples of $\xi_1^2$.
Let us assume we have proven \eqref{201031e7_1} for $d=d_0$, that is, we have established \eqref{201031e7_1} for all $\bq$ depending on $d_0$ variables.
We aim to prove it for $d=d_0+1$, that is, for $\bq$ depending on $d_0+1$ variables.
First of all, we apply Proposition~\ref{prop:skew-lower-bd} and obtain
\begin{equation}
\Gamma_{q, p}(\bq)\ge \max_{0\leq n'\leq n} \Big( (2(d_0+1) - \numvar_{d_0+1,n'}(\bq)) \bigl(\frac{1}{2}-\frac{1}{p}\bigr) - \frac{2(n-n')}{p}\Big)+ (d_0+1)(\frac 1 p-\frac 1 q),
\end{equation}
which is the right hand side of \eqref{201031e7_1} with $d'=d_0+1$.
It remains to prove that
\begin{equation}\label{201031e7_4}
\Gamma_{q, p}(\bq)\ge \max_{d'\leq d_0} \max_{0\leq n'\leq n} \Big( (2d' - \numvar_{d',n'}(\bq)) \bigl(\frac{1}{2}-\frac{1}{p}\bigr) - \frac{2(n-n')}{p}\Big)+ (d_0+1)(\frac 1 p-\frac 1 q).
\end{equation}
Let $H\subset \R^{d_0+1}$ be a linear subspace of dimension $d_0$.
By Lemma~\ref{lem:subspace-lower-bd}, we obtain
\begin{equation}
\Gamma_{q, p}(\bq)\ge \Gamma_{q, p}(\bq|_H)+\frac 1 p-\frac 1 q.
\end{equation}
Now we apply our induction hypothesis to $\Gamma_{q, p}(\bq|_H)$ as $\bq|_H$ depend on $d_0$ variables, and obtain
\begin{equation}\label{201031e7_6}
\Gamma_{q, p}(\bq)\ge \max_{d'\le d_0}\max_{n'\le n}\Big( (2d' - \numvar_{d',n'}(\bq|_H)) \bigl(\frac{1}{2}-\frac{1}{p}\bigr) - \frac{2(n-n')}{p}\Big)+ (d_0+1)(\frac 1 p-\frac 1 q).
\end{equation}
In order to prove \eqref{201031e7_4}, we first take the sup over $H$ in \eqref{201031e7_6} and realize that it suffices to prove
\begin{equation}
\inf_H \numvar_{d', n'}(\bq|_H)\le \numvar_{d', n'}(\bq),
\end{equation}
for every $H$ of dimension $d_0$ and every $d'\le d_0$.
This follows from the definition of $\numvar_{d', n'}$.
\end{proof}

The following example shows that Proposition~\ref{prop:skew-lower-bd} does not by itself always give the correct lower bound for $\Gamma_{q, p}$ when $q>p$.
Let us take the extreme case $q=\infty$.
\begin{example}
\label{ex:1}
Let $d=4$, $n=2$, and
\[
\bq = (\xi_{1}^{2}+\xi_{2}\xi_{4},\xi_{3}\xi_{4})
\]
We have
\[
\numvar_{4,2} = 4, \quad
\numvar_{4,1} = 2, \quad
\numvar_{3,2} = 1,
\]
and all other $\numvar_{d',n'}$ are $0$.
Let $p = 2 + 4n/d = 4$.
Then direct computation shows that $\Gamma_{\infty,p}=9/4$.
However Proposition~\ref{prop:skew-lower-bd} only shows $\Gamma_{\infty, p}\ge 2$.
\end{example}

\section{Proofs of Corollaries~\ref{optimal-l2-decoupling}--\ref{20201021.1.5.}}\label{201031section7}

\subsection{Proof of Corollary~\ref{optimal-l2-decoupling}}

We apply Theorem~\ref{thm:main} with $q=2$ to the tuple of quadratic forms $\bq$, and by \eqref{eq:Gamma-2-p-universal-lower-bd}, we know that \eqref{20201022.1.4} holds true if and only if
\begin{equation}\label{20201022.3.1}
\begin{split}
&\max_{0\leq d'\leq d} \max_{0\leq n'\leq n} \Big( (d'-\numvar_{d',n'} (\bq))\bigl(\frac{1}{2}-\frac{1}{p}\bigr) - \frac{2(n-n')}{p}\Big)
\\&
\leq
\max{\Big(0,d(\frac12-\frac1p)-\frac{2n}{p} \Big)}
\end{split}
\end{equation}
for every $p\ge 2$.
Both sides of \eqref{20201022.3.1} are finite maxima of affine linear functions in $1/p$.
The two arguments of the $\max$ on the right hand side coincide at $p_0:=2+4n/d$.
Hence, \eqref{20201022.3.1} holds for every $p \in [2,\infty]$ if and only if it holds for all $p\in \Set{2,p_{0},\infty}$.

For $p=2$, we have $LHS\eqref{20201022.3.1} = 0 = RHS\eqref{20201022.3.1}$.
For $p=\infty$, we have
\[
LHS\eqref{20201022.3.1} = \max_{0\leq d'\leq d} \max_{0\leq n'\leq n} (d'-\numvar_{d',n'} (\bq))/2 = d/2,
\]
where the maximum is attained at $d'=d$ and $n'=0$, and therefore \eqref{20201022.3.1} holds with equality at $p=\infty$.
For $p=p_{0}$, \eqref{20201022.3.1} is equivalent to
\begin{equation}
\label{eq:10}
\max_{0\leq d'\leq d} \max_{0\leq n'\leq n} \Big( (d'-\numvar_{d',n'} (\bq)) \frac{2n}{dp_{0}} - \frac{2(n-n')}{p_0}\Big) \leq 0.
\end{equation}
A direct calculation shows that \eqref{eq:10} is equivalent to the strong non-degeneracy condition \eqref{equ:1strongly_non_dege}.

\subsection{Proof of Corollary~\ref{201017coro1.3}}
The proof is basically the same as that for Corollary~\ref{optimal-l2-decoupling}.
We apply Corollary~\ref{1015.1.3} to the tuple of quadratic forms $\bq$, and by \eqref{eq:Gamma-p-p-universal-lower-bd}, we know that \eqref{20201021.1.4} holds true if and only if
\begin{equation}\label{20201021.2.1}
\begin{split}
&\max_{d/2 < d'\leq d} \max_{0\leq n'\leq n} \Big( (2d' - \numvar_{d',n'}(\bq)) \bigl(\frac{1}{2}-\frac{1}{p}\bigr) - \frac{2(n-n')}{p}\Big)
\\&
\leq
\max{\Big(d(\frac12-\frac1p),2d(\frac12-\frac1p)-\frac{2n}{p} \Big)}
\end{split}
\end{equation}
for every $p\ge 2$.
The two numbers on the right hand side coincide at $p_0=2+4n/d$.
As in the proof of Corollary~\ref{optimal-l2-decoupling}, \eqref{20201021.2.1} holds for every $p \in [2,\infty]$ if and only if it holds for all $p\in \Set{2,p_{0},\infty}$.
For $p\in \Set{2,\infty}$, the condition \eqref{20201021.2.1} again always holds with equality.
Hence, \eqref{20201021.2.1} holds for every $p \in [2,\infty]$ if and only if it holds at $p=p_0$, which is further equivalent to
\begin{equation}
\label{eq:11}
\max_{d/2 < d'\leq d} \max_{0\leq n'\leq n} \Big( (2d' - \numvar_{d',n'}(\bq)) \frac{2n}{dp_{0}} - \frac{2(n-n')}{p_{0}}\Big)
\leq
\frac{2n}{p_{0}}.
\end{equation}
A direct calculation shows \eqref{eq:11} is equivalent to the non-degeneracy condition \eqref{eq:non_degenerate}.

\subsection{Proof of Corollary~\ref{20201021.1.5.}}
Recall from Corollary~\ref{1015.1.3} that
\[
\Gamma_{p}(\bq)
=
\max_{d/2 < d'\leq d} \max_{0\leq n'\leq n} \gamma_{d',n'}(1/p),
\quad
\gamma_{d',n'}(1/p)
=
(2d' - \numvar_{d',n'}(\bq)) \bigl(\frac{1}{2}-\frac{1}{p}\bigr) - \frac{2(n-n')}{p}.
\]
The functions $\gamma_{d',n'}$ are affine.
For $n' < n$ and arbitrary $d'$, we have $\gamma_{d',n'}(1/2) < 0$.
Moreover, for arbitrary $d'$, we have $\gamma_{d',n}(1/2) = 0$.
For every $p\in (2,\infty)$, the condition \eqref{equ:weakly_non_dege} is equivalent to
\[
\forall d' \in (d/2,d] \quad
\gamma_{d',n}(1/p) \leq d(1/2-1/p).
\]
In particular, if \eqref{equ:weakly_non_dege} fails, then \eqref{20201021.1.21.} fails for any $p_c>2$.

Suppose now that the condition \eqref{equ:weakly_non_dege} is satisfied.
Then, in particular, $\numvar_{d,n}(\bq)=d$, and it follows that
Then, Corollary~\ref{1015.1.3} implies
\begin{equation*}
\Gamma_{p}(\bq)
=
\max\Bigl( d \bigl(\frac{1}{2}-\frac{1}{p}\bigr) ,
\max_{d/2 < d'\leq d} \max_{0\leq n'< n} \gamma_{d',n'}(1/p) \Bigr).
\end{equation*}
Since the latter double maximum is a piecewise affine function of $1/p$ and is strictly negative for $p=2$, we see that there exists $p_c>2$ satisfying \eqref{20201021.1.21.}.
The largest possible $p_c$ is the minimum of solutions $p \in (2,\infty)$ of the equations
\begin{equation}
d\bigl(\frac12-\frac{1}{p}\bigr) = \gamma_{d',n'}(1/p)
\end{equation}
for $d/2 < d' = d-m \leq d$ and $0 \leq n' \leq n-1$.
These solutions are given by the formula
\begin{equation}
p(d',n')=2+\frac{4n-4n'}{2d'-d-\numvar_{d',n'}(\bq)}.
\end{equation}
This shows \eqref{giantnumber}, since the minimum in \eqref{giantnumber} is restricted in such a way as to be taken over numbers in $(2,\infty)$.

We note also that, for $n'=0$ and $m=0$, we have $\numvar_{d,0}(\bq)=0$, which shows that the minimum in \eqref{giantnumber} is taken over a non-empty set, and is at most $2+4n/d$.

\section{Fourier restriction: proof of Corollary~\ref{coro:app_to_restriction}}\label{201031section10}

In this section we prove Corollary~\ref{coro:app_to_restriction}.
The proof is standard, and it relies on an epsilon removal lemma of Tao \cite{MR1666558}, the broad-narrow analysis of Bourgain and Guth \cite{MR2860188} and the decoupling inequalities established in the current paper.
The use of decoupling inequalities in this context is also standard, see for instance Guth \cite{MR3877019}.
As $\bq$ will be fixed throughout the proof,
we will leave out the dependence of the extension function $E^{\bq}g$ on $\bq$ and simply write $Eg$.

Let us begin with the epsilon removal lemma.
In order to prove \eqref{201031e2_5}, it suffices to prove that for every $\epsilon>0$, there exists $C_{d, n, p, \bq, \epsilon}=C_{\epsilon}$ such that
\begin{equation}\label{201106e9_1}
\norm{E_{[0, 1]^d} g}_{L^p(B)}\le C_{\epsilon} \delta^{-\epsilon}\norm{g}_p,
\end{equation}
for every $\delta\le 1, \epsilon>0$, $p>p_{\bq}$ and every ball $B\subset \R^{d+n}$ of radius $\delta^{-2}$.
Here and below, we will leave out the dependence of our implicit constants on $d, n, p$ and $\bq$.
Such a reduction first appeared in \cite{MR1666558}, see also \cite{MR2860188, Kim2017SomeRO}.
For a version of epsilon removal lemmas for manifolds of co-dimension bigger than one, we refer to Section 4 in \cite{2020arXiv200907244G}.\\

In order to prove \eqref{201106e9_1}, we will apply the broad-narrow analysis and the decoupling inequalities in the current paper, together with an induction argument on $\delta$.
Let us assume that we have proven \eqref{201106e9_1} with $\delta'$ in place of $\delta$ for every $1\ge \delta'\ge 2\delta$.
Under this induction hypothesis, we will prove \eqref{201106e9_1}.
Let us begin with one corollary of Proposition~\ref{prop:ball-inflation}.

\begin{corollary}[Multilinear restriction estimate]\label{MRE}
Let $K \in 2^{\N}$ be a dyadic integer and $0 < \delta \leq 1/K$.
Let $\theta>0$ and $\Set{W_j}_{j=1}^{M} \subseteq \Part{1/K}$ be a $\theta$-uniform set of cubes.
Let $B\subset \R^{d+n}$ be a ball of radius $\delta^{-2}$.
Then, for each $2 \leq p < \infty$ and $\epsilon'>0$, we have
\begin{equation}
\begin{split}
\norm[\Big]{\avprod_{j=1}^{M}
|E_{W_j}g|}_{L^{p}(B)}
\le C_{\theta,K,\epsilon'} \delta^{-\gamma({p,\theta,\bq})-\epsilon'} \avprod_{j=1}^{M} \norm{g}_{L^2({W_j})},
\end{split}
\end{equation}
where
\begin{equation}
\gamma({p,\theta,\bq}):= \sup_{0 \leq n' \leq n}\bigg(\frac{2n'}{p} +\big(\frac2p-(1-\theta)\big)\numvar_{n'}(\bq) \bigg).
\end{equation}
\end{corollary}

\begin{proof}[Proof of Corollary~\ref{MRE}]
The proof is essentially via the argument of passing from multi-linear Kakeya estimates to multi-linear restriction estimates as in Bennett, Carbery and Tao \cite{MR2275834}.
Let us first show that
\begin{equation}\label{9.3.ball-inflation}
\begin{split}
\norm[\Big]{\avprod_{j=1}^{M}
|E_{W_j}g|}_{\avL^{p}(B)}
\le C_{\theta,K,\epsilon'} \delta^{-2(\frac{d}{2} - \frac{d+n}{p} + \frac{\kappa((1-\theta)p/2)}{p})-\epsilon'} \avprod_{j=1}^{M} \Big(\sum_{\Box \in \Part[W_j]{\delta}}\norm{E_{\Box}g}_{\avL^{2}(w_B)}^2 \Big)^{1/2},
\end{split}
\end{equation}
for every $\epsilon'>0$.
We take a Schwartz function $\psi$ such that $\psi$ is positive on the ball of radius one centered at the origin, and Fourier transform of $\psi$ has a compact support.
Let us define the function $\psi_{B}(x):=\psi(\delta^2x)$.
By H\"{o}lder's inequality and $L^2$-orthogonality (see for instance \cite[Appenidx B]{MR4031117}), we see that
\begin{equation}
\begin{split}
\norm[\Big]{\avprod_{j=1}^{M}
|E_{W_j}g|}_{\avL^{p}(B)}
&\lesssim
\norm[\Big]{\avprod_{j=1}^{M}
|\psi_B E_{W_j}g|}_{\avL^{p}(B)}
\\&\lesssim
\delta^{-\epsilon'(d+n)/2}
\norm[\Big]{\avprod_{j=1}^{M}
\norm{\psi_BE_{W_j}g}_{\avL^2({B(x,\delta^{-\epsilon'})})}}_{\avL^{p}_{x \in B}}
\\&\lesssim
\delta^{-\epsilon'(d+n)/2}
\norm[\Big]{\avprod_{j=1}^{M}
\Big(\sum_{J \in \Part[W_j]{\delta^{\epsilon'}} }
\norm{\psi_BE_{J}g}_{\avL^2({w_{B(x,\delta^{-\epsilon'})}})}^2 \Big)^{1/2} }_{\avL^{p}_{x \in B}}.
\end{split}
\end{equation}
We apply \eqref{eq:ball-inflation} and $L^2$-orthogonality, and bound the above term by
\begin{equation}
\begin{split}
&\delta^{-\epsilon'(d+n)/2}
\delta^{-\epsilon'(\frac{d}{2} - \frac{d+n}{p} + \frac{\kappa((1-\theta)p/2)}{p})}
\norm[\Big]{\avprod_{j=1}^{M}
\Big(\sum_{J \in \Part[W_j]{\delta^{\epsilon'}} }
\norm{\psi_BE_{J}g}_{\avL^2({w_{B(x,\delta^{-2\epsilon'})}})}^2 \Big)^{1/2} }_{\avL^{p}_{x \in B}}
\\&
\lesssim
\delta^{-\epsilon'(d+n)/2}
\delta^{-\epsilon'(\frac{d}{2} - \frac{d+n}{p} + \frac{\kappa((1-\theta)p/2)}{p})}
\norm[\Big]{\avprod_{j=1}^{M}
\Big(\sum_{J \in \Part[W_j]{\delta^{2\epsilon'}} }
\norm{\psi_BE_{J}g}_{\avL^2({w_{B(x,\delta^{-2\epsilon'})}})}^2 \Big)^{1/2} }_{\avL^{p}_{x \in B}}.
\end{split}
\end{equation}
We repeat this process and obtain
\begin{equation}
\delta^{-\epsilon'(d+n)/2}
\delta^{-2(\frac{d}{2} - \frac{d+n}{p} + \frac{\kappa((1-\theta)p/2)}{p})}
\norm[\Big]{\avprod_{j=1}^{M}
\Big(\sum_{J \in \Part[W_j]{\delta} }
\norm{\psi_BE_{J}g}_{\avL^2({w_{B(x,\delta^{-2})}})}^2 \Big)^{1/2} }_{\avL^{p}_{x \in B}}.
\end{equation}
We rename $\epsilon'(d+n)/2$ by $\epsilon'$, and the above term is bounded by the right hand side of \eqref{9.3.ball-inflation}.

By Plancherel theorem, we see that
\begin{equation}
\norm{E_{\Box}g}_{\avL^{2}(w_B)}
\lesssim \delta^{d}\norm{g}_{L^2(\Box)}.
\end{equation}
Therefore, by \eqref{9.3.ball-inflation}, we obtain that
\begin{equation}
\norm[\Big]{\avprod_{j=1}^{M}
|E_{W_j}g|}_{L^{p}(B)}
\le C_{\theta,K,\epsilon'} \delta^{-2\kappa((1-\theta)p/2)/p-\epsilon'} \avprod_{j=1}^{M} \norm{g}_{L^2({W_j})}.
\end{equation}
It suffices to apply Corollary~\ref{cor:kappa-bound} to bound $\kappa$.
\end{proof}

We let $\theta$ be a small number, which will be determined later.
Its choice depends only on how close $p$ is to $p_{\bq}$.
Therefore the dependence of the forthcoming constants on $\theta$ will also be compressed.
Readers can take $\theta=0$ for convenience.
We define $p_c$ to be the smallest number such that $\gamma(p_c,\theta,\bq)=0$.
More explicitly,
\begin{equation}
p_c=\max_{1 \leq n' \leq n}\bigg(2+\frac{2n'+2\theta\numvar_{n'}(\bq)}{(1-\theta)\numvar_{n'}(\bq)} \bigg).
\end{equation}
To prove \eqref{201106e9_1}, we run the broad-narrow analysis of Bourgain and Guth \cite{MR2860188}, in a way that is almost the same as in the proof of Proposition~\ref{prop:Bourgain-Guth}.
We repeat the proof there until before the step \eqref{201031e5_25}, with $q=p$ and $f_{\Box}$ replaced by $\psi_B E_{\Box} g$ for every dyadic box $\Box$ of side length $\delta$.
Next, instead of summing over all balls $B'$ of radius $K$ in $\R^{d+n}$, we sum over $B'\subset B$, a ball of radius $\delta^{-2}$, and obtain
\begin{equation}\label{201031e5_25zz}
\begin{split}
& \norm{\sum_{\Box\subset [0, 1]^d} \psi_B E_{\Box}g}_{L^p(B)}\le C_{\epsilon'} \sum_{j=0}^d K_j^{\Lambda+2\epsilon'}\Big( \sum_{W\in \calP(1/K_j)} \norm{\psi_B E_W g}_{L^p(w_B)}^p\Big)^{1/p}\\
& + K^d \sum_{1\le M\le K^d} \sum_{\substack{W_1, \dots, W_M\in \calP(1/K)\\ \theta-\text{uniform}}} \Big( \sum_{B'\subset B} \avprod_{j=1}^M\norm{\psi_B E_{W_j} g}^p_{L^p(B')}\Big)^{1/p},
\end{split}
\end{equation}
where
\begin{equation}
\label{eq:lower-dim-dec-exponent_zz}
\Lambda := \sup_{H}\Gamma_{p}(\bq|_H),
\end{equation}
and the sup is taken over all hyperplanes $H \subset \R^d$ that pass through the origin.
Regarding the second term on the right hand side of \eqref{201031e5_25zz}, we notice that each term $|E_{W_j} g|$ is essentially constant on $B'$ and therefore we can apply Corollary~\ref{MRE} and bound it by $C_{\epsilon', K} \delta^{-\epsilon'} \norm{g}_2$, whenever $p>p_c$.
So far we have obtained
\begin{equation}\label{201106e9_12}
\norm{\psi_B E_{[0, 1]^d}g}_{L^p(B)}\le C_{\epsilon'} \sum_{j=0}^d K_j^{\Lambda+2\epsilon'}\Big( \sum_{W\in \calP(1/K_j)} \norm{\psi_B E_W g}_{L^p(w_B)}^p\Big)^{1/p}+ C_{\epsilon', K} \delta^{-\epsilon'} \norm{g}_2,
\end{equation}
for every $\epsilon'>0$ and $p>p_c$.
After arriving at this form, we are ready to apply an inductive argument as the terms on the right hand side of \eqref{201106e9_12} are of the same form as that on the left hand side, with just different scales.
To be precise, we will apply our induction hypothesis to each $\norm{\psi_B E_W g}_{L^p(w_B)}$.
All these terms can be handled in exactly the same way.
Without loss of generality, we take $W=[0, 1/K_j]^d$.
Recall that
\begin{equation}
E_{W}g(x, y)=\int_{W}g(\xi)e(\xi\cdot x+\bq(\xi)\cdot y)d\xi,
\end{equation}
where $x\in \R^d, y\in \R^n$.
We apply the change of variables $\xi\mapsto \xi/K_j$, the induction hypothesis and obtain
\begin{equation}
\norm{\psi_B E_W g}_{L^p(w_B)}\le C C_{\epsilon} \delta^{-\epsilon} K_j^{-d} K_j^{\frac{d+2n}{p}} K_j^{\frac{d}{p}} \norm{g}_{L^p(W)},
\end{equation}
where $C$ is some new large constant that is allowed to depend on $d, n, p$ and $\bq$.
This, together with \eqref{201106e9_12}, implies that
\begin{equation}\label{201106e9_16}
\norm{E_{[0, 1]^d} g}_{L^p(B)}\le CC_{\epsilon'} C_{\epsilon} \delta^{-\epsilon} \sum_{j=0}^d K_j^{\Lambda-d+\frac{2d+2n}{p}+2\epsilon'}\norm{g}_p+ C_{\epsilon', K} \delta^{-\epsilon'} \norm{g}_p,
\end{equation}
for every $\epsilon'>0$.
Recall from \eqref{201030e5_10} that there exists a small number $c=c_{\epsilon'}$ such that
\begin{equation}\label{201030e5_10zzz}
K^c\le K_1\le K_2\le \dots \le K_d\le \sqrt{K}.
\end{equation}
From \eqref{201106e9_16} we see that if $p$ is such that
\begin{equation}\label{201106e9_18}
\Lambda-d+(2d+2n)/p<0,
\end{equation}
then we can pick $\epsilon'$ small enough and $K$ sufficiently large, depending on $\epsilon'$, such that
\begin{equation}
C C_{\epsilon'}K^{\Lambda-d+(2d+2n)/p+2\epsilon'}\le 1/(2(d+1)).
\end{equation}
After fixing $\epsilon'$ and $K$, we see that in order to control the second term in \eqref{201106e9_16}, we just need to set the constant $C_{\epsilon}$ from \eqref{201106e9_1} to be $2C_{\epsilon', K}$ and then we can close the induction step.

Notice that there were two constraints on $p$, including $p>p_c$ and \eqref{201106e9_18}.
Recall the definition of $\Lambda$ in \eqref{eq:lower-dim-dec-exponent_zz}.
One can apply Theorem~\ref{thm:main} and see that
\begin{equation}
\Lambda=\max_{d'\leq d-1} \max_{0\leq n'\leq n} \Big[ (2d' - \numvar_{d',n'}(\bq)) \bigl(\frac{1}{2}-\frac{1}{p}\bigr) - \frac{2(n-n')}{p}\Big].
\end{equation}
Elementary computation shows that
\begin{equation}
p>\max\bigl(p_c, 2+ \max_{m\ge 1}\max_{n'\le n}\frac{4n'}{2m+\numvar_{d-m, n'}(\bq)}\bigr)=\max(p_c, p_{\bq}).
\end{equation}
As $p_c$ is a continuous function depending on $\theta$, to see that we have the range $p>p_{\bq}$, it suffices to show that
\begin{equation}
\max_{1 \leq n' \leq n}\bigg(2+\frac{2n'}{\numvar_{n'}(\bq)} \bigg) \leq 2+ \max_{m\ge 1}\max_{n'\le n}\frac{4n'}{2m+\numvar_{d-m, n'}(\bq)}.
\end{equation}
This inequality follows from
\begin{equation}
2\numvar_{d,n'}(\bq)\geq 2+ \numvar_{d-1,n'}(\bq),
\end{equation}
for every $n'\ge 1$, which holds true because $\numvar_{d,n'}(\bq)>\numvar_{d-1,n'}(\bq)$ as long as $\numvar_{d,n'}(\bq)>0$.
Recall that we assumed $\bq$ is linearly independent, and therefore we indeed have that $\numvar_{d,n'}(\bq)>0$ for every $n'\ge 1$.
This verifies the range $p>p_{\bq}$ and thus finishes the proof of the corollary.

\printbibliography
\end{document}